\documentclass[12pt,leqno,twoside]{amsart}
\usepackage{amsmath,amscd,amssymb,amsfonts,latexsym,mathrsfs,bm,extarrows,graphicx,stackrel}
\usepackage[all,cmtip]{xy}
\usepackage{ulem}
\usepackage{tikz}

\usepackage{srcltx}
\usepackage[colorlinks=true, pdfstartview=FitV, linkcolor=blue, citecolor=blue, urlcolor=blue]{hyperref}
\usepackage{subcaption} 
\usepackage{amsbsy}
\usepackage{amsthm,alltt,dsfont}

\usepackage[margin=1.2in]{geometry}

\usepackage[utf8]{inputenc}
\usepackage{amsthm,soul}
\usepackage{enumitem}
\usepackage{setspace,kantlipsum}
\usepackage{tikz-cd}
\usepackage{textcomp}

\usepackage{colonequals}
\usepackage{mathtools,comment}

\usepackage{graphicx,bm}
\usepackage{mathptmx}

\theoremstyle{plain} 
\newtheorem{thm}{Theorem}[section]
\newtheorem{lemma}[thm]{Lemma}
\newtheorem{proposition}[thm]{Proposition}
\newtheorem{corollary}[thm]{Corollary}

\theoremstyle{definition}
\newtheorem{notation}[thm]{Notation}

\newtheorem{definition}[thm]{Definition}

\newtheorem{example}[thm]{Example}
\newtheorem{remark}[thm]{Remark}

\newcommand{\Z}{\mathbb{Z}}

\newcommand{\N}{\mathbb{N}}

\newcommand{\cB}{\mathcal{B}}

\newcommand{\cD}{\mathcal{D}}

\newcommand{\cF}{\mathcal{F}}
\newcommand{\cG}{\mathcal{G}}
\newcommand{\cH}{\mathcal{H}}

\newcommand{\cJ}{\mathcal{J}}
\newcommand{\cK}{\mathcal{K}}

\newcommand{\cM}{\mathcal{M}}

\newcommand{\cO}{\mathcal{O}}

\newcommand{\cQ}{\mathcal{Q}}

\newcommand{\bA}{\mathbf{A}}

\newcommand{\bC}{\mathbf{C}}
\newcommand{\bD}{\mathbf{D}}

\newcommand{\bQ}{\mathbf{Q}}

\newcommand{\bZ}{\mathbf{Z}}

\newcommand{\sD}{\mathscr{D}}

\renewcommand{\d}{\partial}

\newcommand{\gr}{\mathrm{Gr}}

\newcommand{\DR}{\mathrm{DR}}
\newcommand{\Sing}{\mathrm{Sing}}

\newcommand{\MHM}{\mathrm{MHM}}

\newcommand{\mhm}{\mathrm{MHM}}
\newcommand{\td}{\mathrm{td}}
\newcommand{\MHS}{\mathrm{mHs}}

\begin{document}

\title[Characteristic classes and higher singularities]{Hirzebruch-Milnor classes of local complete intersections, \\ minimal exponent, and applications to higher singularities}

\author{Bradley Dirks}
\address{School of Mathematics, 1 Einstein Drive, Princeton, New Jersey 08540, USA}
\email{bradley.dirks@stonybrook.edu}

\author{Lauren\c{t}iu Maxim}
\address{L. Maxim : Department of Mathematics, University of Wisconsin-Madison, 480 Lincoln Drive, Madison WI 53706-1388, USA, \newline
{\text and} \newline Institute of Mathematics of the Romanian Academy, P.O. Box 1-764, 70700 Bucharest, ROMANIA}
\email{maxim@math.wisc.edu}

\author{Sebasti\'{a}n Olano}
\address{Department of Mathematics, University of Toronto, 40 St. George St., Toronto, Ontario Canada, M5S 2E4}
\email{seolano@math.toronto.edu}

\begin{abstract}
 In this paper we use the deformation to the normal cone and the corresponding Verdier-Saito specialization to define and study (spectral) Hirzebruch-Milnor type homology characteristic classes for local complete intersections.
Our main results describe vanishing properties of these classes in relation to the minimal exponent. As applications, we show how Hirzebruch-Milnor classes of local complete intersections with a projective singular locus can be used to detect higher Du Bois and higher rational singularities.\end{abstract}

%\begin{center}
    \date{\today}
%\end{center}



\maketitle

\tableofcontents

\section{Introduction}
Hirzebruch-Milnor type characteristic classes of complex hypersurfaces (\cite{CMSS,MSS,MSY23}) and of certain complete intersections (\cite{MSS1}) have been introduced in recent years as a Hodge-theoretic homological measure of the complexity of singularities. In the context of hypersurfaces these characteristic classes have been used in \cite{MSS}, resp., \cite{MSY23,MY}, for detecting rational and Du Bois singularities, and resp. higher versions of such singularities. In this paper, we define and study Hirzebruch-Milnor type homology characteristic classes for arbitrary local complete intersection (lci) varieties, and we describe vanishing properties of these classes in relation to the minimal exponent.  For lci varieties with a projective  singular locus, we show how Hirzebruch-Milnor classes can be used to detect higher Du Bois and higher rational singularities.

Rational and Du Bois singularities are important classes of singularities which are intimately connected to birational geometry, and in particular to the minimal model program. It is, for instance, well known that log terminal singularities are rational, while log canonical singularities are Du Bois. Higher analogues of rational and Du Bois singularities of hypersurfaces and, more generally, of local complete intersections, were introduced and studied recently through Hodge theoretic methods in the works of Jung-Kim-Saito-Yoon \cite{JKSY}, Musta\c{t}\u{a}-Olano-Popa-Witaszek \cite{MOPW}, Friedman-Laza \cite{FL22,FL22a}, Musta\c{t}\u{a}-Popa \cite{MP22}, Chen-Dirks-Musta\c{t}\u{a} \cite{CDM}, Dirks \cite{D23}, etc. (See Section \ref{hdbhr} for the relevant definitions). Such singularities are completely described in terms of the minimal exponent $\widetilde{\alpha}$, an important invariant of singularities introduced by Saito \cite{Saito94} in the case of hypersurfaces, and by Chen-Dirks-Musta\c{t}\u{a}-Olano \cite{CDMO} for local complete intersections, as a refinement of the more classical log canonical threshold ($lct$).

In the case of hypersurfaces, rational and Du Bois singularities were  investigated from the viewpoint of Hirzebruch-Milnor characteristic classes in \cite{MSS}, and the results of loc.cit. were recently extended to the study of higher notions of Du Bois and rational singularities in \cite{MSY23,MY}; see Section \ref{ccc} for an overview of these results. The goal of this paper is to investigate higher rational and higher Du Bois singularities of arbitrary local complete intersections via suitably defined Hirzebruch-Milnor classes. 

For defining Hirzebruch-Milnor characteristic classes of an lci subvariety $X$ in a smooth variety $Y$, we make use of: (i) the deformation to the normal cone $C_XY$ of $X$ in $Y$ (cf. Section \ref{defc}), (ii) the corresponding Verdier-Saito specialization \cite{V,Saito90} (recalled in Section \ref{defc}), and (iii) the Brasselet-Sch\"urmann-Yokura Hirzebruch class transformation \cite{BSY}; compare also with \cite{Sc}, where a corresponding Chern class theory was developed. These new Hirzebruch-Milnor classes, denoted (in the un-normalized case) by $M_{y*}(X \subset Y)\in H_*(X)[y]$, are supported on the singular locus of $X$ and, as such, they provide both quantitative and qualitative information about the singularities of the lci subvariety. We also introduce unipotent and resp. non-unipotent versions of these classes, denoted by $[M^{\{1\}}_{y*}(X \subset Y)]$, resp., $[M^{\{\neq 1\}}_{y*}(X \subset Y)]$.
We, moreover, define spectral Hirzebruch-Milnor classes $M^{sp}_{t*}(X \subset Y)$ of lci varieties by taking advantage of the monodromicity of the specialization complex ${\rm Sp_X}(\bQ^H_Y)$, and using the spectral Hirzebruch class transformation of \cite{MSS} (and recalled in Section \ref{ccc}). These are refined versions of the Hirzebruch-Milnor classes, providing global homological generalizations of the Hodge spectrum introduced in \cite{DMS}. For the precise definition and main properties of the (spectral) Hirzebruch-Milnor classes of an lci variety, and for their relation with previously defined notions of such classes of hypersurfaces, see Section \ref{charc}.

Our main results describe vanishing properties of the spectral Hirzebruch-Milnor classes of lci varieties, expressed in terms of the minimal exponent. 
For $\alpha \in \bQ$, we denote the coefficient of $t^\alpha$ in the spectral Hirzebruch-Milnor class $M^{sp}_{t\ast}(X\subset Y)$ by 
 $M^{sp}_{t\ast}(X\subset Y)|_{t^{\alpha}}\in H_{\ast}(\Sing(X))$. 
We then show the following (see Theorem \ref{thm-mainHighCoefficients}).
\begin{thm} \label{thm-mainHighCoefficientsi}
Let $X\subset Y$ be a codimension $r$ local complete intersection in a smooth complex algebraic variety $Y$ such that the log canonical threshold of $X$ satisfies ${\rm lct}(X) > r-1$

Then we have the vanishing
\begin{equation} M^{sp}_{t\ast}(X\subset Y)\vert_{t^{\alpha}}  = 0 \ \text{ for all } \alpha > \dim Y +r -\widetilde{\alpha}(X),\end{equation}
and if the singular locus ${\rm Sing}(X)$ of $X$ is projective, then the converse holds.
\end{thm}

Recall that for a codimension $r$ lci subvariety $X$ of a smooth variety $Y$, its log canonical threshold satisfies the identity ${\rm lct}(X) = \min\{r,\widetilde{\alpha}(X)\}$, where $\widetilde{\alpha}(X)$ is the minimal exponent of $X$ (see Section \ref{minlct}). Hence the assumption ${\rm lct}(X) > r-1$ appearing in the converse statement of Theorem \ref{thm-mainHighCoefficientsi} can equivalently be written as $\widetilde{\alpha}(X) > r-1$. In particular, if $X$ has Du Bois singularities then this condition is automatically satisfied.

\begin{remark} When $X \subseteq Y$ is defined by a regular sequence $f_1,\dots, f_r \in \cO_Y(Y)$, \cite[Theorem 1.1]{CDMO} shows that the minimal exponent  of $X$ can also be computed through the minimal exponent of an associated hypersurface by the following construction: let $g = \sum_{i=1}^r y_i f_i \in \cO_Y(Y)[y_1,\dots, y_r]$ and consider the hypersurface defined by $g$ on $U = Y \times (\bA^r \setminus \{0\})$. Then
\[ \widetilde{\alpha}(X) = \widetilde{\alpha}(g\vert_U).\]
In light of Theorem \ref{thm-mainHighCoefficientsi} (and Theorem \ref{dbrsp} below), it is reasonable to expect that there may be some relation between the characteristic classes for $X$ (as defined in this paper) and those of the hypersurface defined by $g\vert_U$ (as considered, e.g., in \cite{MSS, MSY23}). We do not see a simple way to relate them, at the moment. The main difficulty is that, in comparing the vanishing cycles mixed Hodge module for $g\vert_U$ with the mixed Hodge module in the definition of $M^{sp}_{t\ast}(X\subset Y)$, a Fourier-Laplace transform is involved (see \cite[Proposition 3.4]{D23}). This does not allow for a simple dictionary between the associated graded de Rham pieces of these mixed Hodge modules.
\end{remark}

As a consequence of Theorem \ref{thm-mainHighCoefficientsi}, we have the following result about the Hodge spectrum of an isolated complete intersection singularity, as introduced in \cite{DMS}, with formula \eqref{di} below already obtained in \cite{D23}.

\begin{corollary}\label{cli} 
Assume that the local complete intersection $X \subset Y$ has only one isolated singularity $\{x\}$ and satisfies ${\rm lct}(X) > r-1$. Then, if
$${\rm Sp}(X,x)=\sum_{\alpha \in \bQ} \overline{m}_{\alpha,x} t^\alpha$$ is the reduced spectrum of $X$ at $x$, we have that
 \begin{equation}\label{di} \min\{\alpha \mid \overline{m}_{\alpha,x} \neq 0\} = \widetilde{\alpha}(X) - r +1.\end{equation}
In particular, $X$ has $k$-Du Bois singularities if and only if
\begin{equation}
\overline{m}_{\alpha,x}=0 \ \ \text{for  all} \ \alpha <k+1.
\end{equation}
Moreover, $X$ has $k$-rational singularities if and only if
\begin{equation} 
\overline{m}_{\alpha,x}\ \ \text{for  all} \ \alpha \leq k+1.
\end{equation}
\end{corollary}

To prove Theorem \ref{thm-mainHighCoefficientsi}, we make use of a new formula for the minimal exponent $\widetilde{\alpha}(X)$ of the lci variety $X$, which is proved in Theorem \ref{tLCIQ}. For the lower spectral classes, we have the following result (see Theorem \ref{tlowCoefficients}).
 \begin{thm}\label{tlowCoefficientsi}
Let $X\subset Y$ be a codimension $r$ local complete intersection in a smooth complex algebraic variety $Y$ with ${\rm lct}(X) > r-1$. 

Then we have the vanishing
\[M^{sp}_{t\ast}(X\subset Y)\vert_{t^{\alpha}}=0 \ \text{ for  all } \ \alpha < \widetilde{\alpha}(X) -r +1.\]

Moreover, if ${\rm Sing}(X)$ is projective, then the converse holds.
\end{thm}

As a consequence, we get the following result for the Hodge spectrum of an isolated complete intersection singularity (using the notations from Corollary \ref{cli}) .
\begin{corollary} Assume that the local complete intersection $X \subset Y$ has only one isolated singularity $\{x\}$ and satisfies ${\rm lct}(X) > r-1$. Then 
\[ \max\{\alpha \in \bQ \mid \overline{m}_{\alpha,x} \neq 0\} = \dim(Y) - \widetilde{\alpha}(X).\]
\end{corollary}

Finally, in Section \ref{hshmc} we use the relation between Hirzebruch-Milnor classes and spectral classes to study higher Du Bois and higher rational singularities via vanishing properties of these classes. Moreover, we get a complete homological characterization of 
such higher singularities in the case when the singular locus of the lci variety $X$ is projective.

If, for an integer $p \geq 0$, we denote  by  
$[M_{y\ast }(X\subset Y)]_{p}$
the coefficient of $y^p$ in $M_{y\ast}(X\subset Y)$, and similarly for  $[M^{\{1\}}_{y*}(X \subset Y)]_p$ and  $[M^{\{\neq 1\}}_{y*}(X \subset Y)]_p$, 
then Theorem \ref{thm-mainHighCoefficients} yields the following (see Theorem \ref{cor-higherSings}).
\begin{thm} \label{cor-higherSingsi}
Let $X\subset Y$ be a codimension $r$ local complete intersection in a smooth complex algebraic variety $Y$. If $X$ has $k$-Du Bois singularities, then
\begin{equation}\label{f29i}
\begin{cases}
[M^{\{\neq 1\}}_{y*}(X \subset Y)]_p=0  & \text{for  all} \ p\geq \dim(Y)-k,  \\
[M^{\{1\}}_{y*}(X \subset Y)]_p=0  &  \text{for  all} \  p\geq \dim(Y)+1-k,
\end{cases}
\end{equation}
and the converse holds if ${\rm Sing}(X)$ is projective and ${\rm lct}(X) > r-1$.

If $X$ has $k$-rational singularities, then
\begin{equation} \label{f28i}
[M_{y\ast}(X\subset Y)]_{p}=0 \ \ \text{for  all}  \ p \geq \dim(Y)- k,
\end{equation}
and the converse holds if ${\rm Sing}(X)$ is projective and ${\rm lct}(X) > r-1$. 
\end{thm}

Similarly, Theorem \ref{tlowCoefficients} yields the following (see Theorem \ref{thm-mainLowCoefficients})).

\begin{thm} \label{thm-mainLowCoefficientsi}
Let $X\subset Y$ be a codimension $r$ local complete intersection in a smooth complex algebraic variety $Y$. If $X$ has $k$-Du Bois singularities, then
\begin{equation}\label{f40i}
[M_{y*}(X \subset Y)]_p=0 \  \text{for  all} \ p\leq k.
\end{equation}
If $X$ has $k$-rational singularities, then \eqref{f40i} holds and, moreover,
\begin{equation} \label{f41i}
[M_{y\ast}^{\{1\}}(X\subset Y)]_{k+1}=0.
\end{equation}

If ${\rm lct}(X) > r-1$ and ${\rm Sing}(X)$ is projective, then the converse implications hold.
\end{thm}

{\bf Convention.} Our characteristic classes are defined for (certain) mixed Hodge modules, by using the De Rham complexes of the underlying filtered $\cD$-modules. Throughout this paper, we work with {\it right} filtered $\sD$-modules. 
\bigskip

{\bf Acknowledgment.} \  The project began at the workshop "Higher Du Bois and higher rational singularities" at the American Institute of Mathematics. The authors thank AIM for providing a supportive and mathematically rich environment. We would also like to thank Debaditya Raychaudhury, J\"org Sch\"urmann and Rosie Shen for useful discussions.  BD was supported by the National Science Foundation under Grant No. DMS-1926686. LM acknowledges support from the Simons Foundation and from the project ``Singularities and Applications'' - CF 132/31.07.2023 funded by the European Union - NextGenerationEU - through Romania's National Recovery and Resilience Plan.

%%%%%%%%%%%%%%%%%%%

\section{Higher Du Bois and higher rational singularities}\label{hdbhr}
In this section we recall the notions of higher Du Bois and higher rational singularities, and relations between them.

Let us begin by recalling that for a reduced complex algebraic variety $X$ there are two generalizations of the classical De Rham complex, namely:
\begin{itemize}
\item[(i)]  the {De Rham complex} $(\Omega^{\bullet}_X,F)$ of K\"ahler differentials,
\item[(ii)]  the {Du Bois complex} $(\underline{\Omega}^{\bullet}_X,F)$; see \cite{duBois}.
\end{itemize}
Moreover, there is a natural morphism of filtered complexes
\begin{equation}
(\Omega^{\bullet}_X,F) \to (\underline{\Omega}^{\bullet}_X,F),
\end{equation}
which is a filtered quasi-isomorphism if $X$ is smooth.

\medskip

In order to define higher Du Bois singularities, we introduce the following notation: for $p \geq 0$, set 
$$ \uuline{\Omega}^{p}_X:=gr_{-p}^F(\underline{\Omega}^{\bullet}_X)[p] \in D^b_{\rm coh}(X)$$
(where we use the convention for increasing Hodge filtrations).
Note that if $X$ is smooth, then $\uuline{\Omega}^{p}_X \simeq {\Omega}^p_X$.
Moreover, if $X$ has only quotient or toroidal singularities, then $$\uuline{\Omega}^{p}_X \simeq \widehat{\Omega}^p_X:=j_* \Omega^p_{X_{\rm reg}},$$ the $p$-th Zariski sheaf, for $j\colon X_{\rm reg} \hookrightarrow X$ the inclusion of the smooth locus.
\begin{definition}\cite{JKSY, MOPW}
For $k\geq 0$, the complex algebraic variety $X$ is said to have {\it $k$-Du Bois singularities} if the 
induced morphism
$$\Omega^p_X \to \uuline{\Omega}^p_X $$
 is an isomorphism in $D^b_{\mathrm{coh}}(X)$  for all $0\leq p \leq k$. 
\end{definition}
Note that when $k=0$, this recovers the usual notion of Du Bois singularities.

\medskip

Higher versions of rational singularities were introduced by Friedman-Laza in \cite{FL22a}.
\begin{definition}\cite{FL22a} 
(1) Assume $X$ is irreducible, with $\mu:(\widetilde{X},D) \to (X, X_{\rm sing})$ a log resolution of singularities.
Say that $X$ has {\it $k$-rational singularities} if the natural morphism
\[ \Omega^p_X \to  R\mu_{\ast}\Omega^p_{\widetilde{X}}(\log D)\]
is an isomorphism for all $0\leq p \leq k$.\\ 
(2) An arbitrary variety $X$ has $k$-rational
singularities if all its connected components are irreducible, with $k$-rational singularities.
\end{definition}
When $k=0$, this recovers the usual notion of rational singularities.

One has the following result relating the above notions of singularities (see also \cite{Kovacs99} for the case $k=0$).

\begin{thm}
Assume $X$ is a locally complete intersection in a smooth variety $Y$.
Then:
\begin{itemize}
\item[(a)] (Musta\c{t}\u{a}-Popa  \cite{MP22}, Friedman-Laza \cite{FL22a})
\begin{center}
$X$ is $k$-rational $\Longrightarrow$ $X$ is $k$-Du Bois.
\end{center} 
\item[(b)] (Musta\c{t}\u{a}-Popa \cite{MP22}, Chen-Dirks-Musta\c{t}\u{a} \cite{CDM})
\begin{center}
$X$ is $k$-Du Bois $\Longrightarrow$ $X$ is  $(k-1)$-rational.
\end{center}
\end{itemize}
\end{thm}

%%%%%%%%%%%%%%%%%%%%%%%%%

\section{Higher singularities of hypersurfaces} 
We recall here motivating results which relate higher singularities to characteristic classes in the case of complex hypersurfaces, see \cite{MY, MSY23} for more details. For simplicity, we restrict here to the case of globally defined hypersurfaces, but see \cite{MSY23} for arbitrary hypersurfaces.

\subsection{Characterization via minimal exponents}\label{sec: k-Du Bois singularity via minimal exponents}
We first recall the characterization of higher Du Bois and higher rational singularities in terms of the minimal exponent. Let $Y$ be a smooth complex algebraic variety and let $f$ be a holomorphic function on $Y$ such that $X\colonequals f^{-1}(0)$ is reduced. The {\it minimal exponent} $\tilde{\alpha}_{f}$ is defined to be the smallest root of $b_f(-s)/(-s+1)$, where $b_f(s)$ is the Bernstein-Sato polynomial associated to $f$. Recall that $b_f(s)=s+1$ if, and only if, $X$ is smooth (in which case the minimal exponent is $\infty$, by convention). The minimal exponent is a refinement of the log canonical threshold $lct(f)$, which satisfies $lct(f)=\min \{1, \tilde{\alpha}_{f}\}$. Moreover, as shown in \cite[Theorem 0.4]{Saito94}, in our setting we always have that $\tilde{\alpha}_{f} \leq \frac{\dim Y}{2}.$

We collect below results from \cite[Theorem 1]{JKSY}, \cite[Theorem 1.1]{MOPW} and \cite[Theorem E]{MP22}.
\begin{thm}\label{thm: k du Bois singularity function via minimal exponent}
Let $Y$ be a smooth complex algebraic variety and let $f$ be a holomorphic function on $Y$ such that $X\colonequals f^{-1}(0)$ is reduced. Then
\begin{itemize}
\item $X$ has $k$-Du Bois singularities if and only if $\tilde{\alpha}_{f} \geq  k + 1$.
\item $X$ has $k$-rational singularities if and only if $\tilde{\alpha}_{f} >  k + 1$.
\end{itemize}
\end{thm}

%%%%%%%%%%%%%%%%%%%
\subsection{Characterization via the $V$-filtration and vanishing cycles}
Throughout this paper, we use right filtered $\sD$-modules. For example, if $Y$ is a connected smooth variety, then the trivial Hodge module $\bQ^H_Y[\dim Y]$ has underlying filtered $\sD_Y$-module $(\omega_Y,F)$ where
\[ F_{-\dim Y-1}\omega_Y = 0,\quad F_{-\dim Y} \omega_Y = \omega_Y.\]

With $Y, f, X$ as above, 
denote the graph embedding of $f:Y \to \bC$ by 
 \[i_f:Y\to Y\times \bC_t, \ x\mapsto (x,f(x)),\]
 where $t$ is the coordinate on $\bC$.  Denote by $(i_f)_{+}$  the direct image functor for filtered $\sD$-modules.
Recall that on a quasi-unipotent regular holonomic $\sD_Y$-module $\cM$ on $Y$ one has the $V$-filtration associated to $f$, which we index decreasingly and which satisfies the property that $t \d_t + \alpha$ is nilpotent on $\gr^{\alpha}_V\cM_f$, where 
 \[\cM_f\colonequals (i_f)_{+}\cM.\]
 
 Our convention here is that, if $\cM^\ell$ is the corresponding left $\sD_Y$-module which satisfies $\cM = \omega_Y \otimes_{\cO_Y} \cM^\ell$, then
 \[ V^\alpha \cM_f= \omega_{Y\times \mathbf C} \otimes_{\cO} (V^\alpha \cM_f^\ell),\]
 where $\d_t t - \alpha$ is nilpotent on ${\rm Gr}_V^\alpha (\cM^\ell)$.
 
If $M$ is a  mixed Hodge module with underlying filtered $\sD$-module $(\cM,F)$ with increasing filtration $F$, the filtered $\sD$-module underlying the vanishing cycle mixed Hodge module $\varphi^H_fM$ with the decomposition 
 \[ \varphi^H_fM = \varphi^H_{f,\neq 1}M \oplus \varphi^H_{f,1}M\]
 is 
 \begin{align*}
 \varphi^H_f\cM &=\bigoplus_{0\leq \alpha< 1}\gr^{\alpha}_V\cM_f,\\
 \varphi^H_{f,\neq 1}\cM&=\bigoplus_{0<\alpha <1} \gr^{\alpha}_V\cM_f, \quad  \varphi^H_{f,1}\cM=\gr^{0}_V\cM_f,
 \end{align*}
 with 
 \begin{align*}
 F_k \varphi^H_{f,\neq 1}\cM&=\bigoplus_{0<\alpha <1} F_{k-1}\gr^{\alpha}_V\cM_f, \quad F_k \varphi^H_{f,1}\cM=F_{k}\gr^{0}_V\cM_f.
 \end{align*}

For the filtered right $\sD_Y$-module $(\omega_Y,F)$ underlying the constant mixed Hodge module $\bQ^H_Y[\dim Y]$, we let 
\[ \cB_f\colonequals (i_f)_{+}\omega_Y \cong \omega_Y\otimes_{\bC}\bC[\d_t]\delta_f ,\]
where $\delta_f$ is a formal symbol which helps to keep track of the $\cD$-action:
\[  (\mu \delta_f)t = f \mu \delta_f, \quad (\mu \delta_f) D = (\mu D)\delta_f + D(f) \mu \d_t \delta_f, \, \mu \in \omega_Y, D \in {\rm Der}_{\mathbf C}(\cO_Y),\]

The Hodge filtration is defined by
\[ F_{k-\dim Y} \cB_f\colonequals \bigoplus_{0\leq \ell\leq k} \omega_Y \d_t^{\ell} \delta_f,\]
so that
\begin{equation}\label{eqn: associated graded of Hodge filtrations on Bf}
\gr^F_{k-\dim Y}\cB_f=\omega_X \d_t^k\delta_f , \quad \forall k\in \N.
\end{equation}

The following result was proved in \cite{MY, MSY23}.
\begin{thm}\label{t32} In the above notations,
\[  
\tilde{\alpha}_{f} \geq  k + 1 \iff F_{k-\dim Y} \gr^\alpha_V \mathcal{B}_f=0, \ \ 0 \leq \alpha <1.  
\] 
\[  \tilde{\alpha}_{f} >  k + 1 \iff  
 \begin{cases} 
F_{k-\dim Y} \gr^\alpha_V \mathcal{B}_f=0,  \ \  0 \leq \alpha <1,  \\ 
F_{k+1-\dim Y}\gr^0_V\mathcal{B}_f=0. 
\end{cases} 
\]
\end{thm}

\begin{corollary}
Let $Y$ be a smooth complex algebraic variety and let $f\colon Y \to \bC$ be a non-constant holomorphic function such that $X=f^{-1}(0)$ is reduced. Then, for any $k\in \N$,  
$X$ has $k$-Du Bois singularities 
if and only if
\[\begin{cases}
F_{k+1-\dim Y} \varphi^H_{f,\neq 1}(\bQ^H_Y[\dim Y])=0,\\
F_{k - \dim Y}\varphi^H_{f,1}(\bQ^H_Y[\dim Y])=0.
\end{cases} \]
Similarly, $X$ has $k$-rational singularities
if and only if
\[ F_{k+1-\dim Y}\varphi^H_{f}(\bQ^H_Y[\dim Y])=0.\]
\end{corollary}

%%%%%%%%%%%%%%%%%%%%%%

\subsection{Characterization via spectral Hirzebruch classes}\label{ccc}
For a complex algebraic variety $X$, we let $\MHM(X)$ denote the abelian category of algebraic mixed Hodge modules on $X$, cf. \cite{Saito90}. We use $H_i(X)$ in place of either $H_{2i}^{BM}(X;\bQ)$ or $CH_i(X)_\bQ$, and let $K_0(X):=K_0(Coh(X))$. We next recall the definition of Hirzebruch classes, and of their spectral versions.

First, we introduce the motivic Chern class transformation
\[
 \DR_y: K_0(\mhm(X)) \to  K_0(X)[y,y^{-1}],
\]
which is defined as follows. If $X$ is smooth and $M\in \mhm(X)$, let $(\cM,F_{\bullet}\cM)$ be the underlying filtered $\sD_X$-module of $M$, and set
\[
\begin{split}
\DR_y[M] &\colonequals \sum_{p}  \left[\gr^F_{-p}\DR(\cM)\right] \cdot(-y)^p  \\
&= \sum_{p,i} (-1)^i \left[\cH^i \gr^F_{-p}\DR(\cM)\right] \cdot(-y)^p \in K_0(X)[y,y^{-1}],
\end{split}
\]
with $\gr^F_{-p}\DR(\cM)$ the graded parts of the de Rham complex of $\cM$ with respect to the induced Hodge filtration.
The definition of $\DR_y$ extends to the case when $X$ is singular by using locally defined closed embeddings into smooth varieties, since the graded quotient cohomology sheaves $\cH^i\gr^F_{p}\DR(\cM)$ are independent of local embeddings and are $\cO_X$-modules. Furthermore, it can also be extended to complexes $M^\bullet \in D^b\mhm(X)$ by applying it to each cohomology module $H^iM^\bullet\in \mhm(X)$, $i \in \Z$. 

\begin{definition}\cite{BSY,S}
The (un-normalized) {\it Hirzebruch class transformation}  is defined by 
\begin{align*}
T_{y\ast}: K_0(\mhm(X)) &\to H_*(X)[y, y^{-1}] \\
[M^\bullet] &\mapsto \td_{\ast} (\DR_y[M^\bullet]),
\end{align*}
where $\td_{\ast}:K_0(X) \to H_\ast(X)$
is the 
Baum--Fulton--MacPherson Todd class transformation \cite{BFM}, which is linearly extended over $\bZ[y,y^{-1}]$. A normalized Hirzebruch class transformation $\widehat{T}_{y\ast}$ is then defined by precomposing $T_{y\ast}$ with a homological Adams-type operation $\{\Psi_k\}_{k \geq 0}$, which on a class $\gamma \in H_k(X)$ acts by $\Psi_k(\gamma)=(1+y)^{-k} \cdot \gamma$.
\end{definition}

\begin{remark}
Over a point, the transformations $T_{y\ast}$ and $\widehat{T}_{y\ast}$ coincide with the Hodge polynomial homomorphism $\chi_y:K_0(\MHS^p) \to \bZ[y^{\pm 1}]$ defined on the Grothendieck group of (graded) polarizable mixed Hodge structures by
$$\chi_y([H]):=\sum_p \dim \gr^p_F(H \otimes \bC) \cdot (-y)^p,$$
for $F$ the (decreasing) Hodge filtration of $H \in \MHS^p$. Here we use the well-known equivalence of categories $\MHM(pt) \simeq \MHS^p$.
\end{remark}

By applying the above Hirzebruch class transformations to a vanishing cycle complex, one gets (an un-normalized version of) the {\it Hirzebruch--Milnor classes} introduced in \cite{CMSS,MSS1}. Let  
$Y$ be a smooth complex algebraic variety and let $f\colon Y\to \bC$ be a non-constant holomorphic function. Set $X\colonequals f^{-1}(0)$. Let $\varphi^H_f$ be the vanishing cycle functor of mixed Hodge modules. We will work with the shifted functor \begin{equation}\label{conv} \varphi_f:=\varphi^H_f[1],\end{equation} whose underlying functor on constructible complexes is the Deligne vanishing cycle functor.
\begin{definition}\label{hms}
The (un-normalized) {\it Hirzebruch--Milnor class} of the hypersurface $X=f^{-1}(0)$ is defined by:
\begin{equation}\label{HMc}
\begin{split}
    M_{y\ast}&(X):=T_{y\ast}(\varphi_f\bQ^H_Y)\\
    &=\td_\ast \left( \sum_p \left[\gr^F_{-p}\DR(\varphi_f\bQ^H_Y)\right] \cdot(-y)^p \right) \in H_*(\Sing(X))[y],
    \end{split}
\end{equation} 
where $\bQ_Y^H$ is the shifted constant Hodge module on $Y$. (Here we use the functoriality of $T_{y*}$ for proper maps, and the fact that $\varphi_f\bQ^H_Y$ is supported on $\Sing(X)$ to view $M_{y*}(X)$ as a localized class on $\Sing(X)$.)
One can define similarly a normalized version $\widehat{M}_{y\ast}(X)$ of the Hirzebruch--Milnor class by using instead $\widehat{T}_{y\ast}$ in \eqref{HMc}. \end{definition}

\begin{example}\label{ex37}
If the hypersurface $X$ has only isolated singularities, then $\varphi_f\bQ_Y$ is supported only at these singular points, and we get (cf. \cite[Example 3.6]{CMSS}):
$$M_{y\ast}(X)=\widehat{M}_{y\ast}(X)=\sum_{x \in \Sing(X)} \chi_y(i_x^* \varphi_f\bQ^H_Y)=\sum_{x \in \Sing(X)} \chi_y([\widetilde{H}^*(F_x;\bQ)]),$$
where $i_x:\{x\}\hookrightarrow X$ is the point inclusion, and $F_x$ denotes the Milnor fiber of the isolated hypersurface singularity germ $(X,x)$.
\end{example}

The Hirzebruch class transformation $\widehat{T}_{y\ast}$ has been lifted in \cite{MSS} to a spectral version, i.e., a characteristic class version of the Hodge spectrum. Here we work (as in \cite{MY}) mainly with an un-normalized version of this class.
As in the previous section, let us first assume that $X$ is a smooth complex algebraic variety. Let $\mhm(X,T_s)$ be the abelian category of mixed Hodge modules $M$ on $X$ which are endowed with an action of $T_s$ of finite order. For $(M, T_s) \in \mhm(X,T_s)$ with $T_s^e=\mathrm{Id}$, let $(\cM,F_{\bullet}\cM)$ be the underlying filtered $\sD_X$-module of $M$.  There is a canonical decomposition 
$$(\cM,F_{\bullet})=\sum_{\lambda\in \mu_e} (\cM_\lambda,F_{\bullet}),$$
such that $T_s=\lambda\cdot \mathrm{Id}$ on $\cM_\lambda$, where $\mu_e=\{\lambda \in \bC \mid \lambda^e=1\}$. 
With these notations, 
we define the {\it spectral motivic Chern class} of $(M, T_s)$ by
\begin{align*}
\DR_t[M, T_s] := \sum_{p,\lambda} [\gr^F_{-p}\DR(\cM_\lambda)] \cdot  t^{p+\ell(\lambda)} \in K_0(X)[t^{ 1/e},t^{-1/e}],
\end{align*}
where $ \ell(\lambda)$ is the unique number in $[0,1)$ such that $\exp({2\pi i \ell(\lambda)})=\lambda$. As in the case of the motivic Chern class transformation, the above definition extends to the case when $X$ is singular and also to complexes $(M^{\bullet}, T_s)$ of mixed Hodge modules on $X$ endowed with the action of a finite order automorphism $T_s$.  

Let $K^{mon}_0(\mhm(X))$ denote the Grothendieck group of mixed Hodge modules on $X$ endowed with a finite order automorphism. 
\begin{definition}\cite{MSS}
The (un-normalized) {\it spectral Hirzebruch transformation} is defined by
\[ T^{sp}_{t\ast}\colon  K^{mon}_0(\mhm(X)) \to \bigcup_{e\geq 1} H_\ast(X)\left[t^{1/e}, t^{-1/e}\right]\]
\[
[M^{\bullet},T_s]\mapsto \td_{\ast}\left(\DR_t[M^{\bullet}, T_s]\right)  ,
\]
where $\td_\ast$ is as before the Todd class transformation. A normalized spectral Hirzebruch class transformation $\widehat{T}^{sp}_{t\ast}$ is then defined by precomposing $T^{sp}_{t\ast}$ with a homological Adams-type operation $\{\Psi_k\}_{k \geq 0}$, which on a class $\gamma \in H_k(X)$ acts by $\Psi_k(\gamma)=(1-t)^{-k} \cdot \gamma$.
\end{definition}

\begin{remark}\label{r39}
Over a point, the transformations $T^{sp}_{t\ast}$ and $\widehat{T}^{sp}_{t\ast}$ coincide with the Hodge spectrum homomorphism (see \cite{DL, GLM})
$${\rm Sp}' :K^{mon}_0(\MHS) \to \bZ[\bQ] \simeq \bigcup_{e\geq 1} \bZ[t^{\pm 1/e}]$$ defined on the Grothendieck ring of the abelian category $\MHS^{mon}$ of mixed Hodge structures endowed with a finite order automorphism by 
$${\rm Sp}'([H]):=\sum_{p, \lambda} ( \dim \gr^p_F H_{\bC,\lambda}) \cdot t^{p+\ell(\lambda)}$$
for any mixed Hodge structure $H \in \MHS^{mon}$ with an automorphism $T$ of finite order, 
where $F$ is the (decreasing) Hodge filtration on $H_\bC=H \otimes \bC$, and $H_{\bC,\lambda}=\ker(T-\lambda)$ is the eigenspace of $T$ with eigenvalue $\lambda=\exp(2\pi i \ell(\lambda))$ and $ \ell(\lambda)\in \bQ \cap [0,1)$. In \cite{Saito91}, whose notation we follow for the operation on mixed Hodge structures, this is called the ``dual spectrum''.

Note that the Hodge polynomial $\chi_y$ is obtained from the Hodge spectrum ${\rm Sp}'$ by forgetting the $T$-action and substituting $t=-y$.
\end{remark}

Let us now restrict to the case of globally defined hypersurfaces. Let $Y$ be a smooth complex algebraic variety and let $f\colon Y \to \bC$ be a non-constant holomorphic function with $X=f^{-1}(0)$. Let 
\[\varphi_f\bQ^H_Y \in \mhm(X)[-\dim(X)]\]
be the (shifted) vanishing cycle mixed Hodge module, and let $T_s$ be the semi-simple part of the monodromy. Then 
one can introduce the following un-normalized version of the spectral Hirzebruch class from \cite{MSS}.
\begin{definition}\label{definition: spectral Hirzebruch-Milnor class}
In the above notations, the (un-normalized) {\it localized spectral Hirzebruch--Milnor class} of the hypersurface $X$ is defined by:
\[ {M^{sp}_{t*}(X):=T^{sp}_{t*}(\varphi_f\bQ^H_Y,T_s)} \in H_*(\Sing(X))[t^{1/ord(T_s)}].\]
We can similarly define a normalized version $\widehat{M}^{sp}_{t*}(X)$ of this class by using $\widehat{T}^{sp}_{t\ast}$ instead.
\end{definition}

\begin{example}\label{ex311}
If the hypersurface $X$ has only isolated singularities, we get as in Example \ref{ex37} that 
$$M^{sp}_{t*}(X)=\widehat{M}^{sp}_{t*}(X) = \sum_{x \in \Sing(X)} {\rm Sp}'(X,x) \coloneq \sum_{x \in \Sing(X)} {\rm Sp}'(i_x^* \varphi_f\bQ^H_Y)$$
where the action is given by the semi-simple part of the Milnor monodromy action for each isolated hypersurface singularity germ $(X,x)$.
\end{example}

\begin{remark} It is important to note that what we call ${\rm Sp}'(X,x)$ differs from ${\rm Sp}'(X,x)$ as defined in \cite{Saito91}, by a sign $(-1)^{\dim X}$. Indeed, in \textit{loc. cit.}, the dual spectrum at $x\in X$ is defined by ${\rm Sp}'(i_x^*\varphi_f^H(\mathbf Q_Y^H[\dim Y]))$.
\end{remark}

\medskip

As already shown in \cite{MSS}, the spectral Hirzebruch--Milnor class $M^{sp}_{t\ast}(X)$ of $X=f^{-1}(0)$ is sensitive to $X$ having rational or Du Bois singularities. The results of \cite{MSS} were further extended in \cite{MY, MSY23} to higher notions of these singularities as follows. 
We make the following.
\begin{notation}
For $\alpha \in \bQ$, we denote the coefficient of $t^\alpha$ in the spectral Hirzebruch-Milnor class $M^{sp}_{t\ast}(X)$ of $X$ by 
\[ M^{sp}_{t\ast}(X)|_{t^{\alpha}}\in H_{\ast}(\Sing(X)).\] 
\end{notation}
Then the following holds.
\begin{thm}\cite{MY,MSY23}\label{dbrsp}
With the above notations, if $X$ has $k$-Du Bois singularities then $M^{sp}_{t*}(X)|_{t^\alpha} = 0$ for all $\alpha < k+1$, 
with the converse implication being true if $\Sing(X)$ is projective.

Moreover, if $X$ has $k$-rational singularities then $M^{sp}_{t*}(X)|_{t^\alpha} = 0$ for all $\alpha \leq k+1$, 
with the converse implication being true if $\Sing(X)$ is projective.
\end{thm}

For some examples when the assumption that $\Sing(X)$ is projective is satisfied, see \cite{Saito23}. As noted in \cite[Remark 1]{MSY23}, one cannot replace projectivity by compactness in the assumption of the converse assertions in Theorem \ref{dbrsp}. On the other hand, one cannot omit the compactness assumption, since the Chow groups of affine varieties are usually small.

\begin{remark}\label{rem314}
When the hypersurface $X=f^{-1}(0)$ has only isolated singularities, Theorem \ref{dbrsp} provides necessary and sufficient conditions for these singularities to be $k$-Du Bois, resp., $k$-rational, formulated in terms of the classical Hodge spectrum. More precisely, since the spectral Hirzebruch--Milnor class is a characteristic class version of the notion of Hodge spectrum ${\rm Sp'}(X,x)=\sum_{\alpha \in \bQ} n_\alpha(f,x) t^\alpha$ of a hypersurface singularity germ $(X,x)$, cf. Remark \ref{r39}, Theorem \ref{dbrsp} reads in the case of $X$ with only isolated singularities as follows: $X$ has only $k$-Du Bois singularities if, and only if, $n_\alpha(f,x) = 0$ for all $\alpha < k+1$, and all $x\in \Sing(X)$. Similarly, 
$X$ has only $k$-rational singularities if, and only if, $n_\alpha(f,x) = 0$ for all $\alpha \leq k+1$, and all $x\in \Sing(X)$. (Note that in the case of isolated hypersurface singularities, the two notions of spectrum used in this paper, ${\rm Sp}$ and ${\rm Sp'}$, coincide (up to sign) by the self-duality of the mixed Hodge structure on the Milnor cohomology.)
\end{remark}

Let us note that, if for an integer $p \geq 0$ we denote by  
\[
[M_{y\ast }(X)]_{p}  \in H_\ast(\Sing(X))
\]
the coefficient of $y^p$ in $M_{y\ast}(X)$, then we have the identity
\begin{equation}\label{MHsp} [M_{(-y)\ast}(X)]_{p}=\bigoplus_{\alpha \in \bQ\cap[0,1)} M^{sp}_{t\ast}(X)|_{t^{p+\alpha}},\end{equation}
In particular, as a consequence of Theorem \ref{dbrsp}, one has the following interpretation of $X$ being $k$-Du Bois without involving the monodromy action (generalizing the case $k=0$ considered in \cite[Theorem 5]{MSS}, 
which in turn extended a result of Ishii \cite{I} proved in the case of isolated singularities).
\begin{corollary}\label{thm: main 1}
With the above notations, if $X$ has $k$-Du Bois singularities, then
\[[M_{y\ast}(X)]_{p}=0 \in H_{\ast}(\Sing(X)) \textrm{ for all \ $p \leq k$}.\]
If $\Sing(X)$ is projective, then the converse is true. 
In particular, if $X$ has only isolated singularities, then $X$ has $k$-Du Bois singularities if, and only if, $\gr^p_F H^n(F_x;\bC)=0$ for all $p \leq k$ and all $x \in \Sing(X)$, where $n=\dim X$ and $F$ is the Hodge filtration on the cohomology of the Milnor fiber $F_x$ of $f$ at $x \in \Sing(X)$.
\end{corollary}

Similarly, using the decomposition $\varphi_f=\varphi_{f,1}\oplus \varphi_{f,\neq 1}$, with $\varphi_f\colonequals \varphi_f^H[1]$ as before, we also define the {\it unipotent Hirzebruch--Milnor class of $X$} by
\begin{equation}\label{unih} M^{\{1\}}_{y\ast}(X):=T_{y\ast}(\varphi_{f,1}\bQ^H_Y) \in H_\ast(\Sing(X))[y],\end{equation}
and we denote by $[M^{\{1\}}_{y\ast}(X)]_{k+1}$  the coefficient of $y^{k+1}$ in \eqref{unih}. Then Theorem \ref{dbrsp} implies the following generalization of \cite[(4.2.6)]{MSS} where the case $k=0$ was considered. 
\begin{corollary}\label{thm: main 2}
With the above notations, if $X$ has $k$-rational singularities, then
\begin{center} 
$[M_{y\ast}(X)]_{p}=0 \in H_{\ast}(\Sing(X))$ for all \ $p \leq k$,  and  $[M^{\{1\}}_{y\ast}(X)]_{k+1}=0  \in H_{\ast}(\Sing(X))$.
\end{center}
If $\Sing(X)$ is projective, then the converse is true.
\end{corollary}

For formulations of such results in terms of the normalized (spectral) Hirzebruch-Milnor classes, see \cite{MSY23}.

%%%%%%%%%%%%%%%%%%%%%%%%%

\section{Higher singularities of local complete intersections}

In this section we use of the deformation to the normal cone and the Verdier-Saito specialization to define (spectral) Hirzebruch-Milnor classes for an arbitrary lci subvariety $X$ of a smooth complex algebraic variety $Y$. These classes are global versions of the (dual) Hodge spectrum introduced in \cite{DMS}.
We study these classes in relation to the minimal exponent, and relate them to higher notions of rational and Du Bois singularities of lci varieties.

\subsection{Minimal exponent}\label{minlct}
Recall that for an arbitrary closed subscheme $X$ of the smooth variety $Y$, one can define a Bernstein--Sato polynomial $b_X(s)\in \bQ[s]$, extending the classical notion from the case of hypersurfaces (see \cite{BMS}). 

If $X \subseteq Y$ is a (nonempty) local complete intersection of pure codimension $r$, 
then $(s+r)$ divides $b_X(s)$,  see \cite[Proposition 6.1]{CDMO}. 
By analogy with the definition of the minimal exponent in the case of hypersurfaces, one can then define $\tilde{\alpha}(X)$ to be the negative of the largest root of $b_X(s)/(s+r)$ (with the convention that this is infinite if $b_X(s)=(s+r)$). As shown in \cite{D23}, $\tilde{\alpha}(X)$ coincides indeed with the minimal exponent of the local complete intersection $X$, as introduced in \cite{CDMO}. By definition, if $Y = \bigcup_{i\in I} U_i$ is an open cover, then
\begin{equation}\label{eqMinExpCover} \widetilde{\alpha}(X) = \min_{i\in I} \widetilde{\alpha}(X\cap U_i).\end{equation}

\begin{remark} \label{rmk-LCT} In the local complete intersection setting, we see that $\widetilde{\alpha}(X)$ is a refinement of the log canonical threshold of $(Y,X)$, just like in the hypersurface setting. Recall that the log canonical threshold is defined as follows: if $\mathfrak a\subseteq \cO_Y$ is the coherent ideal sheaf defining $X$ in $Y$, then
\[ {\rm lct}(Y,X) = {\rm lct}(\mathfrak a) = \min\{\lambda > 0 \mid \cJ(\mathfrak a^{\lambda}) \neq \cO_X\},\]
where $\cJ(\mathfrak a^{\lambda})$ is the multiplier ideal of $X$ inside $Y$ with exponent $\lambda$, defined through numerical data on a log resolution of $(Y,X)$.

Then ${\rm lct}(Y,X) = \min\{r,\widetilde{\alpha}(X)\}$, similarly to the hypersurface case. Below, we suppress $Y$ from the notation and write ${\rm lct}(X)$. In several statements below, the assumption ${\rm lct}(X) > r-1$ appears, and by this equality, it could equivalently be written $\widetilde{\alpha}(X) > r-1$, though we prefer to state it in terms of the log canonical threshold as that invariant is more well-known. By Theorem \ref{tLCIHighSings} below, if $X$ has Du Bois singularities then this condition is automatically satisfied.
\end{remark}

One then has the following generalization of Theorem \ref{thm: k du Bois singularity function via minimal exponent}.
\begin{thm}\label{tLCIHighSings} \cite{CDMO,CDM}
If $Y$ is a smooth, irreducible variety and $X$ is a local complete intersection closed subscheme of $Y$, of pure codimension $r$, then:
\begin{itemize}
\item $X$ has $k$-Du Bois singularities $\iff$ $\tilde{\alpha}(X) \geq k+r$
\item $X$ has $k$-rational singularities $\iff$ $\tilde{\alpha}(X) > k+r$.
\end{itemize}
\end{thm}

\subsection{Relation to the $V$-filtration} 
As above, let $X\subseteq Y$ be a local complete intersection subvariety of pure codimension $r$. Fix an open cover $Y = \bigcup_{j\in J} U_j$ such that, for any $j\in J$ with $X\cap U_j \neq \emptyset$ (which, below we will assume is true for all $j\in J$, for ease of notation), we have $X\cap U_j \subseteq U_j$ is a complete intersection subvariety, meaning its ideal sheaf is defined by a regular sequence in $\mathcal O_{U_j}(U_j)$.

For some $j\in J$, let $X\cap U_j \subseteq U_j$ be defined by a regular sequence $f_1,\dots, f_r \in \cO_{U_j}(U_j)$. Again, for the filtered $\sD_{U_j}$-module $(\omega_{U_j},F)$ underlying the constant mixed Hodge module $\bQ^H_{U_j}[\dim U_j]$, we let 
\[ \cB_{f,U_j} \colonequals (i_f)_{+}\omega_{U_j} \cong \omega_{U_j}\otimes_{\bC}\bC[\d_{t_1},\dots, \d_{t_r}]\delta_f ,\]
with Hodge filtration defined by
\[ F_{k-\dim Y} \cB_{f,U_j} \colonequals \bigoplus_{0\leq |\alpha| \leq k} \omega_{U_j} \d_t^{\alpha}\delta_f.\]

This module admits a $V$-filtration along $(t_1,\dots, t_r)$ as studied in \cite{BMS}. This is an exhaustive decreasing, discretely $\bQ$-indexed filtration $V^\bullet \cB_{f,U_j}$ such that if $\theta_t = \sum_{i=1}^r t_i \d_{t_i}$, then $\theta_t+\lambda$ is nilpotent on ${\rm Gr}_V^\lambda  \cB_{f,U_j}$. An important difference between the higher codimension case and the case of a hypersurface is that, in general, ${\rm Gr}_V^\lambda \cB_{f,U_j}$ need not be a holonomic (or even coherent) $\sD_{U_j}$-module. Nonetheless, we have the following analogue of Theorem \ref{t32} (which, in one sense, is the definition of the minimal exponent for complete intersection subvarieties).

\begin{thm} \label{tLCI} \cite{CDMO} Let $X \subseteq Y$ be a local complete intersection subvariety of pure codimension $r$ and $Y = \bigcup_{j\in J} U_j$ be an open cover as above. For any $\mu \leq r$, we have
\[ \widetilde{\alpha}(X) \geq \mu \iff \text{ for all } j\in J, \text{ we have } F_{-\dim Y} \gr_V^{\chi}\cB_{f,U_j} = 0 \text{ for all } \chi < \mu.\]

Moreover, for any $\lambda \in (0,1]$ and $k\in \bZ_{\geq -1}$, we have the equivalence
\[ \widetilde{\alpha}(X) \geq r+k + \lambda \iff \text{ for all } j\in J,\text{ we have } F_{k+1-\dim Y} \gr_V^{r-1+\chi}\cB_{f,U_j} = 0 \text{ for all } \chi < \lambda.\]
\end{thm}

%%%%%%%%%%%%%%%%%%%%%%

\subsection{Deformation to the normal cone, specialization}\label{defc}
Let us now assume that $\ell: X \hookrightarrow Y$ is the embedding of a 
local complete intersection variety $X$ of codimension $r$ 
in a smooth complex algebraic variety $Y$. 
Then $\ell$ gives rise to an embedding $$X \times \{0\} \hookrightarrow Y \times \{0\} \hookrightarrow Y \times \bC.$$
Consider the blowup ${\rm Bl}_{X\times \{0\}} (Y \times \bC)$ of $Y\times \bC$ along $X \times \{0\} $ (with ${\rm Bl}$ denoting the blowup), with exceptional divisor $\mathbb{P}(C \oplus \bf{1})$, for $C=C_XY$ 
the normal cone of $X$ in $Y$. (Note that $C$ is in fact an algebraic vector bundle of rank $r$.)
The second factor projection $Y\times \bC \to \bC$ induces a flat morphism 
$${\rm Bl}_{X\times \{0\}} (Y \times \bC) \to \bC$$
whose zero fiber is given by $${\rm Bl}_{X}Y \cup_{\mathbb{P}(C)} \mathbb{P}(C \oplus \bf{1}).$$
Let
\[
\widetilde{Y}={\rm Bl}_{X\times \{0\}} (Y \times \bC) \setminus {\rm Bl}_{X \times \{0\}} (Y\times \{0\}),
\]
and consider the induced morphism $h:\widetilde{Y} \to \bC$, whose zero fiber is $h^{-1}(0)=C$ and general fiber $h^{-1}(t)=Y \times \{ t \}$, $t \in \bC^*$.
Let us set $Y\times \bC^*=:\widetilde{Y}^*$. Then we have the {\it deformation to the normal cone}, explained by the following diagram:
\begin{equation}\label{diag}
\xymatrix{
C_XY \ar[d] \ar@{^(->}[r]^i & \widetilde{Y} \ar[d]_h & \ar@{^(->}[l]_j \ar[d] \widetilde{Y}^*  \\
\{ 0 \} \ar@{^(->}[r] & \bC & \ar@{^(->}[l] \bC^* 
}
\end{equation}
To the above diagram, one can associate nearby and vanishing cycle functors $\psi_h$, resp., $\varphi_h$. The nearby cycle functor $\psi_h$ is used to define the {\it Verdier specialization} functor along $X$, see \cite{V, Saito90}. However, for our purpose, the vanishing cycles of $h$ and the associated Hirzebruch-Milnor class of $C_XY$ will turn out to be related to the singularities of $X$.

With $j\colon \widetilde{Y}^*=Y\times \bC^* \to \widetilde{Y}$ the open (affine) inclusion, denote by $$q\colon \widetilde{Y}^* \to Y$$ the projection map. Then for $M \in \MHM(Y)$, the {\it Verdier specialization} of $M$ is defined as
\[
{\rm Sp}_X(M):=\psi_h(j_*q^*M) \in \MHM(C_XY),
\]
with $\psi_h=\psi_h^H[1]$ corresponding to Deligne's nearby cycle functor as in \eqref{conv}, see \cite[Section 2.30]{Saito90}. Recall that the functor ${\rm Sp}_X$ commutes with duality $\bD$, and it induces the identity on $\MHM(X)$, meaning that for $$s=s_0:X \hookrightarrow  C_XY$$ the inclusion of the zero section of $C_XY$ and $M\in D^b\MHM(Y)$, there is an isomorphism
\begin{equation}\label{zero} s^*{\rm Sp}_X(M) \to \ell^* M \end{equation} in $D^b\MHM(X)$. Here $\ell\colon X \hookrightarrow Y$ is the inclusion map, and we consider the induced functor ${\rm Sp}_X$ on derived categories.

Let  $$\rho:\widetilde{Y}\to Y\times \bC \to Y$$
be the composition of the (restriction of the) blowup with the projection map. Then $\rho \circ j=q$, whence $$j^*(j_*q^*M)\simeq q^*M\simeq j^*(\rho^*M).$$
Since for $N \in \MHM(\widetilde{Y})$, the nearby cycle $\psi^H_h(N)$ only depends on $j^*N$, it follows that
\[
{\rm Sp}_X(M)\simeq \psi_h(\rho^*M).
\]

Next consider the distinguished triangle
$$i^* \overset{sp}{\longrightarrow} \psi_h \overset{can}{\longrightarrow} \varphi_h \overset{[1]}{\longrightarrow},$$
with $i\colon C_XY \hookrightarrow \widetilde{Y}$ as in \eqref{diag}, and apply it to $\rho^*M$ to get a triangle
\begin{equation}\label{f1}
i^*\rho^*M \overset{sp}{\longrightarrow} {\rm Sp}_X(M)  \overset{can}{\longrightarrow} \varphi_h(\rho^*M)  \overset{[1]}{\longrightarrow}
\end{equation}

Finally, for $M=\bQ^H_Y[\dim(Y)]$, \eqref{f1} becomes the following triangle in $D^b\MHM(C_XY)$:
\begin{equation}\label{f2}
\bQ^H_{C_XY}[\dim(X)+r] \overset{sp}{\longrightarrow} {\rm Sp}_X(\bQ^H_Y[\dim(Y)])  \overset{can}{\longrightarrow}\varphi_h(\bQ^H_{\widetilde{Y}}[\dim(\widetilde{Y})-1])  \overset{[1]}{\longrightarrow}
\end{equation}

\begin{remark}
Let us next note that since $X$, and hence $C_XY$, is a local complete intersection,  the complex $\bQ^H_{C_XY}[\dim(X)+r]$ is a single mixed Hodge module on $C_XY$. Similarly, since $X$ is a local complete intersection, $\widetilde{Y}$ is a local complete intersection as well, and so $\bQ_{\widetilde{Y}}^H[\dim(\widetilde{Y})]$ is a mixed Hodge module. Since the functor $\varphi^H_h=\varphi_h[-1]$ preserves mixed Hodge modules, it follows that $\varphi_h(\bQ^H_{\widetilde{Y}}[\dim(\widetilde{Y})-1])$ is a single mixed Hodge module as well. 
Hence the three objects in the above triangle \eqref{f2} are mixed Hodge modules on $C_XY$. By taking the long exact sequence of cohomology objects associated to  \eqref{f2},  we obtain the following short exact sequence in $\MHM(C_XY)$:
\begin{equation}\label{f3}
0\longrightarrow \bQ^H_{C_XY}[\dim(X)+r] \overset{sp}{\longrightarrow} {\rm Sp}_X(\bQ^H_Y[\dim(Y)]) \overset{can}{\longrightarrow} \varphi_h(\bQ^H_{\widetilde{Y}}[\dim(\widetilde{Y})-1])  \longrightarrow 0.
\end{equation}
\end{remark}

\begin{remark}\label{rmkVerdier}
It is important to highlight a difference between the definition of the monodromy on ${\rm Sp}_X(\bQ^H_Y[\dim(Y)])$ and the one introduced in \cite{V}, as Verdier's results are referenced in later sections. Throughout this paper, we consider only the monodromy induced by the vanishing cycle $\psi_h$ of $h$ (as described in \cite[p.25]{Sc-book}), which is the inverse of the one defined via the isomorphism $\iota_{2\pi i}$ in \cite[\textsection 9]{V}. 
\end{remark}

\begin{remark}\label{supp}
We note here that $\varphi_h(\bQ^H_{\widetilde{Y}}[\dim(\widetilde{Y})-1])$ is supported on the restriction of the normal bundle $C_XY$ to $\Sing(X)$, e.g., on $\Sing(X) \times \bC^r$ if $X$ is a global complete intersection (since in this case $C_XY$ is a trivial rank $r$ vector bundle on $X$). 
For this, we note that if $X$ (and hence $C_XY$) is smooth, then $\widetilde{Y}$ is smooth and the shifted complex $\varphi_h \bQ^H_{\widetilde{Y}}$ is supported on $\Sing(h^{-1}(0))=\emptyset$, i.e., it is the zero complex (in fact, $h$ is  in this case a smooth morphism, as one can check in local coordinates). So when $X$ is smooth, the injective map $sp$ in \eqref{f3} is an isomorphism. Then the assertion follows from the localization property (SP0) of the specialization functor (cf. \cite{V}), which induces a similar property for $\varphi_h(\bQ^H_{\widetilde{Y}}[\dim(\widetilde{Y})-1])$.
\end{remark}

%%%%%%%%%%%%%%%%%%%
\subsection{Minimal exponent, redux}
In this section, we prove the following result, which is a consequence of Theorem \ref{tLCI} above.

\begin{thm} \label{tLCIQ} Let $X\subseteq Y$ be a local complete intersection subvariety of pure codimension $r$. Assume $X$ satisfies ${\rm lct}(X)> r-1$. We have the equality
\begin{equation}\label{eq5} \widetilde{\alpha}(X) = r -1 + \min\{p+\lambda \mid \lambda \in (0,1], {\rm Gr}^F_{p-\dim Y}{\rm DR}\varphi_{h,\lambda}(\bQ^H_{\widetilde{Y}}[\dim(\widetilde{Y})-1]) \neq 0\},\end{equation}
where $\varphi_{h,\lambda}$ is short for $\varphi_{h,\exp(-2\pi i \lambda)}$.
\end{thm}

Before proving the theorem, we provide more details on the underlying filtered $\cD$-module of $\varphi_{h,\lambda}$.
Since both sides of \eqref{eq5} can be defined by taking the minimum over an open cover $Y = \bigcup_{j\in J} U_j$, we can assume that $X \subseteq Y$ is a complete intersection subvariety, defined by $f_1,\dots, f_r \in \mathcal O_Y(Y)$. 
Let $i_f \colon Y \to Y \times \bC^r_t$ be the graph embedding along $(f_1,\dots, f_r)$, with fiber coordinates $t_1,\dots, t_r$ on the affine space. As above, denote $\cB_f = (i_f)_*(\bQ_Y^H[\dim Y])$. This has underlying $\cD$-module $\cB_f = \bigoplus_{\alpha \in \N^r} \omega_Y \partial_t^\alpha \delta_f$, where $\delta_f$ is a formal symbol as above. The Hodge filtration is given by
\[ F_{p-\dim Y} \cB_f = \bigoplus_{|\alpha| \leq p} \omega_Y \partial_t^\alpha \delta_f,\]
and this module admits a Kashiwara-Malgrange $V$-filtration along $t_1,\dots, t_r$, which we denote $V^\bullet \cB_f$. At this point, we observe that $F_{p-\dim Y}\cB_f$ is a coherent $\cO_Y$-module, and so any $\cO_Y$-subquotient is also coherent.

In the above notations, we view ${\rm Sp}_X(\bQ_Y^H[\dim(Y)])$ as a mixed Hodge module on $Y \times \bC^r_z$, and it has underlying filtered $\cD$-module given by (see \cite{BMS, CD})
\[ {\rm Sp}(\cB_f) = \bigoplus_{\chi \in \bQ} {\rm Gr}_V^{\chi}(\cB_f),\]
\[ F_p {\rm Sp}(\cB_f) = \bigoplus_{\chi \in \bQ} F_p {\rm Gr}_V^{\chi}(\cB_f), \text{ where } F_p {\rm Gr}_V^\chi(\cB_f) = \frac{F_p V^\chi \cB_f}{F_p V^{>\chi}\cB_f}.\]
Here the action of $\cD_{Y\times \bC^r} = \cD_Y\langle z_1,\dots, z_r, \partial_{z_1},\dots, \partial_{z_r}\rangle$ is defined by
\[  [m]P = [m] P \text{ for all } m\in {\rm Gr}_V^\chi(\cB_f), \ P \in \cD_Y,\]
\[ [ m]z_i = [mt_i] \in {\rm Gr}_V^{\chi+1}(\cB_f) \text{ for all } m\in {\rm Gr}_V^\chi(\cB_f),\]
\[ [m] \d_{z_i} = [m\d_{t_i}] \in {\rm Gr}_V^{\chi-1}(\cB_f) \text{ for all } m\in {\rm Gr}_V^{\chi}(\cB_f).\]

Let $\cK_f = \bigcap_{i=1}^r \ker(\partial_{t_i} \colon {\rm Gr}_V^r(\cB_f) \to {\rm Gr}_V^{r-1}(\cB_f)) \subseteq {\rm Gr}_V^r(\cB_f)$. The canonical inclusion $\bQ_{C_X Y}^H[\dim(X) + r] \xrightarrow[]{sp} {\rm Sp}_{X,1}(\bQ_Y^H[\dim(Y)])$ corresponds to the inclusion
\[ \cK_f[z_1,\dots, z_r] \hookrightarrow \bigoplus_{\ell \in \bZ} {\rm Gr}_V^\ell(\cB_f),\]
\[ [m]z^\alpha \mapsto [mt^\alpha] \in {\rm Gr}_V^{r+|\alpha|}(\cB_f).\]
This map is strictly filtered, so that the Hodge filtration on $\cK_f[z_1,\dots, z_r]$ is determined by that on ${\rm Sp}_{X,1}(\bQ_Y^H[\dim(Y)])$. The quotient module is, by definition, $\varphi_{h,1}(\bQ_{\widetilde{Y}}^H[\dim(\widetilde{Y})-1])$, and so the Hodge filtration on this vanishing cycles module is determined by the Hodge filtration on ${\rm Sp}_{X,1}(\bQ_Y^H[\dim(Y)])$. Note that, for $\lambda \in (0,1)$, we have a filtered isomorphism
\begin{equation} \label{eq-nonUnipQ} \bigoplus_{\ell \in \Z} {\rm Gr}_V^{\lambda+\ell}(\cB_f,F) \cong (\varphi_{h,\lambda}(\bQ_{\widetilde{Y}}^H[\dim(\widetilde{Y})-1]),F).
\end{equation}

Although this formula looks similar to the definition of vanishing cycles in terms of $V$-filtrations, we remark that the right hand side is not computed in terms of the $V$-filtration of the hypersurface defined by $h$ in $\widetilde{Y}$, because $\widetilde{Y}$ is singular. The resemblance to the definition of vanishing cycles in terms of $V$-filtrations can be especially confusing here, because neither side has a shift of the Hodge filtration, but we hope no confusion will arise.

We will need some results concerning monodromic mixed Hodge modules on $Y\times \bC^r$. Recall that these are modules whose underlying filtered $\cD$-module splits into generalized eigenspaces for the Euler operator $\theta = \sum_{i=1}^r z_i \partial_{z_i}$. In other words, if $\cM$ is monodromic, then $\cM = \bigoplus_{\chi \in \bQ} \cM^\chi$ where $\cM^\chi = \bigcup_{j>1} \ker((\theta + \chi)^j)$. The Hodge filtration also satisfies $F_p \cM = \bigoplus_{\chi \in \bQ} F_p \cM^\chi$, where $F_p \cM^\chi = \cM^\chi \cap F_p \cM$.
It is an easy application of \cite[Thm. 1]{CD} that, for $\chi > r$, we have equality $F_p \cM^{\chi} = \sum_{i=1}^r F_p \cM^{\chi-1}z_i $. This has the consequence that if $\cM$ is a (filtered direct summand of a) monodromic module such that $F_p \cM = \bigoplus_{\chi > r-1} F_p \cM^\chi$, then
\begin{equation}\label{eq-TestZero} F_p \cM = 0 \text{ if and only if } F_p \cM^{r-1+\lambda} = 0 \text{ for all } \lambda \in (0,1].\end{equation}

Moreover, if we take $\lambda \in (0,1]$ and $\sigma = \min\{p \in \bZ \mid F_p \cM^{\lambda+\bZ} \neq 0\}$ and if we assume $F_\sigma \cM^{\lambda + \bZ} = \bigoplus_{\ell \in \bZ_{\geq 0}} F_\sigma \cM^{r-1+\lambda+\ell}$, then the following natural map is surjective:
\begin{equation}\label{eq-surjectLowestHodge} (F_\sigma \cM^{r-1+\lambda})[z_1,\dots, z_r] \to F_\sigma \cM^{\lambda + \bZ}\end{equation}

In fact, by \cite[Lemma 2.6]{D23} that natural map is an isomorphism in this case. Thus, if $F_\sigma \cM^{\lambda+\bZ} = \bigoplus_{\ell \geq 0} F_\sigma \cM^{r-1+\lambda+\ell}$, we have an isomorphism
\begin{equation}\label{eq-isoLowestHodge}\pi^*(F_\sigma \cM^{r-1+\lambda}) = (F_\sigma \cM^{r-1+\lambda})[z_1,\dots, z_r] \cong F_\sigma \cM^{\lambda+\bZ},  \quad \pi \colon Y \times \bC^r_z \to Y,\end{equation}
where $\pi^*$ is the pullback of $\cO$-modules and $F_\sigma \cM^{r-1+\lambda}$ is a coherent $\cO_Y$-module.

If $s\colon Y \to Y \times \bC^r$ is the zero section and if $\cM$ is monodromic such that $F_\sigma \cH^r s^!(\cM) = 0$ and $F_{\sigma}\cM = \bigoplus_{\chi \geq r-1} F_\sigma \cM^\chi$, then we have an isomorphism
\[ \pi^*(F_\sigma \cM^{r-1}) = F_\sigma \cM^{r-1}[z_1,\dots, z_r] \cong F_\sigma \cM^{\bZ}.\]
Indeed, $F_\sigma \cH^r s^!(\cM) =0$ implies that $F_\sigma \cM^r = \sum_{i=1}^r F_\sigma \cM^{r-1}z_i$ by \cite[Theorem 2]{CD}, and so one can use the same argument as above.

We can now prove the theorem. 
\begin{proof}[Proof of Theorem \ref{tLCIQ}] Assume $\widetilde{\alpha}(Z) = r + \lambda + p$ for some $\lambda \in (0,1]$ and $p \in \bZ_{\geq -1}$. By Theorem \ref{tLCI}, this means
\[ F_{p+1-\dim Y} {\rm Gr}_V^{r-1+\chi} \cB_f = 0 \text{ for all } \chi < \lambda\]
and we have non-vanishing
\[ \begin{cases} F_{p+1-\dim Y} {\rm Gr}_V^{r-1+\lambda} \cB_f \neq 0 & \lambda \in (0,1) \\ F_{p+2-\dim Y} {\rm Gr}_V^{r-1} \cB_f \neq 0 & \lambda = 1 \end{cases}.\]

To prove the claim, we first show that $F_{p-\dim Y} \varphi_h(\bQ_{\widetilde{Y}}^H[\dim(\widetilde{Y})-1]) = 0$. Using Theorem \ref{tLCI} (specifically, the inequality $\widetilde{\alpha}(Z) \geq r+ p$), we see that $F_{p-\dim Y} {\rm Gr}_V^\chi(\cB_f) = 0$ for all $\chi < r$. Moreover, as $F_{p+1 - \dim Y} {\rm Gr}_V^{r-1}(\cB_f) = 0$, we know that
\[ F_{p-\dim Y} \cK_f = F_{p-\dim Y} {\rm Gr}_V^r(\cB_f),\]
meaning that the $r$th monodromic Hodge piece vanishes: $F_{p-\dim Y} \varphi_{h,1}(\bQ_{\widetilde{Y}}^H[\dim(\widetilde{Y})-1])^r = 0$. This tells us by the equivalence \eqref{eq-TestZero} that $F_{p-\dim Y} \varphi_h(\bQ_{\widetilde{Y}}^H[\dim(\widetilde{Y})-1]) =0$.

We prove the claim in two cases. First, if $\lambda \in (0,1)$, we have by Equation \eqref{eq-nonUnipQ} that $F_{p+1-\dim Y} \varphi_{h,\lambda}(\bQ_{\widetilde{Y}}^H[\dim(\widetilde{Y})-1]) \neq 0$. Moreover, for all $\mu \in (0,\lambda)$, we have 
\[F_{p+1-\dim Y} \varphi_{h,\mu}(\bQ_{\widetilde{Y}}^H[\dim(\widetilde{Y})-1]) = 0,\]
using the equivalence \eqref{eq-TestZero} for the monodromic module $\varphi_{h,\mu}(\bQ_{\widetilde{Y}}^H[\dim(\widetilde{Y})-1])$. This proves the claim in this case.

Finally, if $\lambda = 1$, we know that $F_{p+1-\dim Y} \cB_f \subseteq V^r \cB_f$ and that $F_{p+2-\dim Y}\cB_f \not \subseteq V^{>r-1}\cB_f$. Note that $F_{p+2-\dim Y}\cB_f = F_{p+1-\dim Y} \cB_f + \sum_{i=1}^r (F_{p+1-\dim Y} \cB_f)\partial_{t_i}$, by definition, so the fact that this is not contained in $V^{>r-1}\cB_f$ implies that
\[ F_{p+1-\dim Y} \cK_f \subsetneq F_{p+1-\dim Y} {\rm Gr}_V^r(\cB_f),\]
or, in other words, that $F_{p+1-\dim Y} \varphi_{h,1}(\bQ_{\widetilde{Y}}^H[\dim(\widetilde{Y})-1]) \neq 0$. As we have already observed that $F_{p+1-\dim Y} \varphi_{h,\lambda}(\bQ_{\widetilde{Y}}^H[\dim(\widetilde{Y})-1]) = 0$ for all $\lambda \in (0,1)$ in this case, we are done.
\end{proof}

To end this subsection, we show how to glue the local isomorphisms in equation \eqref{eq-isoLowestHodge} above. The difficulty is the following: for a mixed Hodge module $M$ on a singular variety $X$, we have well-defined objects ${\rm Gr}^F_p {\rm DR}(M)$ in $D^b_{\rm coh}(\cO_X)$ for all $p \in \bZ$. In particular, if we take the first $p$ where this complex is non-zero, we get a well-defined coherent sheaf (the lowest Hodge piece of the filtered right $\cD$-module underlying $M$ in any local smooth embedding). 

Now, if $M$ is such that there exists a collection of local smooth embeddings so that $M$ is monodromic in each embedding, we want this result to hold for the individual monodromic pieces, too.

\begin{lemma}\label{lem-globalization} Let $X\subseteq Y$ be a local complete intersection of pure codimension $r$. For any $\lambda \in \bQ$ and $p\in \bZ$, there is a coherent $\cO_X$-module ${\rm Gr}^F_p {\rm Gr}_V^\lambda(\cB_{X,Y})$ with the property that if $U \subseteq Y$ is an open subset such that $X\cap U \subseteq U$ is a complete intersection defined by a regular sequence $f_1,\dots, f_r \in \cO_{U}(U)$, then there is an isomorphism
\[ {\rm Gr}^F_p {\rm Gr}_V^\lambda(\cB_{X,Y})\vert_U \cong {\rm Gr}^F_p {\rm Gr}_V^\lambda(\cB_f).\]

Moreover, there exists a coherent $\cO_X$-submodule ${\rm Gr}^F_p \cK_{X}$ of ${\rm Gr}^F_p {\rm Gr}_V^r(\cB_{X,Y})$ such that, for any local choice of regular sequence as above, there is an isomorphism
\[ ({\rm Gr}^F_p \cK_{X})\vert_U \cong {\rm Gr}^F_p \cK_f.\]
\end{lemma}
\begin{proof} We cover $Y= \bigcup_{i\in I} U_i$ such that $X\cap U_i \subseteq U_i$, if non-empty, is a complete intersection defined by a regular sequence $f_1,\dots, f_r \in \cO_{U_i}(U_i)$ for all $i \in I$. By replacing $Y$ with $U_i \cap U_j$ for $i \neq j$, we want to construct $\cO_Y$-linear isomorphisms
\[ {\rm Gr}^F_p {\rm Gr}_V^\lambda(\cB_f) \cong {\rm Gr}^F_p {\rm Gr}_V^\lambda(\cB_g)\]
for two regular sequences $(f) = f_1,\dots, f_r, (g) = g_1,\dots, g_r$ defining $X$ inside $Y$. Moreover, we want these isomorphisms to satisfy a cocycle condition if we had a third regular sequence $(h) = h_1,\dots, h_r$ defining $X$.

We follow \cite[Remark 4.8]{CDMO}, so we write $g_i = \sum a_{ij} f_j$ for $1\leq i \leq r$. This gives an $r\times r$-matrix $A = (a_{ij})$ whose determinant $\det(A)$ does not vanish on any point of $X$ (as both $(f), (g)$ are regular sequences). By replacing $Y$ with the non-vanishing locus of $\det(A)$, we have an automorphism
\[u_A \colon Y \times \mathbf C^r_t \to Y \times \mathbf C^r_\eta, (y,\underline{t}) \mapsto (y, \sum a_{ij} t_j),\]
so that if $i_f$ is the graph embedding along $f$ and $i_g$ is the graph embedding along $g$, then $i_g = u \circ i_f$. Thus, there is an $(F,V)$-bifiltered $\cO_Y$-linear isomorphism $\tau_{f,g}^A \colon \cB_f \cong \cB_g$, but where the $\cD_{Y \times \mathbf C^r}$-action on $\cB_f$ is twisted by $u$. Note that the twisted action preserves the $\cO_Y$-module structure.

As $\cB_f = \bigoplus \omega_Y \partial_t^\alpha \delta_f$, it suffices to explain for each $\mu \in \omega_Y$ where the element $\mu \partial_t^\alpha \delta_f$ gets sent. The relation $\eta_i = \sum_{j=1}^r a_{ij} t_j$ gives that $\partial_{t_i} = \sum_{j=1}^r a_{ji} \partial_{\eta_j}$, which shows that the Euler operators are identified: $\theta_t = \theta_\eta$. Then the isomorphism is given by
\[ \mu \partial_t^\alpha \delta_f \mapsto \mu \prod_{i=1}^r \left(\sum_{j=1}^r a_{ji} \partial_{\eta_j}\right)^{\alpha_i}\delta_g.\]

Restricting to ${\rm Gr}^F_p{\rm Gr}_V^\lambda$, we get $\cO_Y$-linear isomorphisms
\begin{equation} \label{eq-gluingIso} {\rm Gr}^F_p {\rm Gr}_V^\lambda(\cB_f) \cong {\rm Gr}^F_p {\rm Gr}_V^\lambda(\cB_g).\end{equation}

We must check that these isomorphisms are independent of the choice of matrix $A$. Let $A'$ be another matrix such that $A(f) = A'(f) = (g)$. Then $(A-A')(f) = 0$, so by the regular sequence property (for example, exactness of the Koszul complex), we see that each entry of $A-A'$ lies in $(f_1,\dots, f_r)$. Using that $ {\rm Gr}^F_p {\rm Gr}_V^{\lambda}(\cB_f)$ is scheme-theoretically supported on $X$, we see then that $\tau_{f,g}^A = \tau_{f,g}^{A'}$ on these objects. We write $\tau_{f,g}$ for those isomorphisms.

If $h_1,\dots, h_r$ is another regular sequence defining $X$ in $Y$, then we can write $(h) = B \cdot (g)$ for some matrix $B$ whose determinant $\det(B)$ doesn't vanish on $X$. Moreover, we have $(h) = BA (f)$ and we can replace $Y$ by the non-vanishing locus of $\det(AB) = \det(A)\det(B)$. In this case, we have $u_{AB} = u_B \circ u_A$ as automorphisms of $Y \times \mathbf C^r$ which gives $(F,V)$-bifiltered isomorphisms
\[ \tau^{BA}_{f,h} = \tau^B_{g,h} \circ \tau^A_{f,g}: \cB_f \to \cB_g \to \cB_h,\]
and so the induced isomorphisms
\[ \tau_{f,h} = \tau_{g,h} \circ \tau_{f,g} \colon {\rm Gr}^F_p {\rm Gr}_V^\lambda(\cB_f) \cong {\rm Gr}^F_p {\rm Gr}_V^\lambda(\cB_g) \cong {\rm Gr}^F_p {\rm Gr}_V^\lambda(\cB_h)\]
satisfy the cocycle condition.

We conclude from this that if $X\subseteq Y$ is an arbitrary local complete intersection subvariety and we take an open cover $Y = \bigcup_{i\in I} U_i$ such that $X\cap U_i \subseteq U_i$ is a complete intersection for all $i$, then the sheaves ${\rm Gr}^F_p {\rm Gr}_V^{\lambda}(\cB_{f,U_i})$ glue together to give $\cO_Y$-coherent sheaves
\[{\rm Gr}^F_p {\rm Gr}_V^{\lambda}(\cB_{X,Y}) \text{ scheme-theoretically supported on } X.\]

Finally, note that $\tau_{f,g}^A$ maps $F_p \cK_f = \bigcap_{i=1}^r \ker(\partial_{t_i} \colon F_p {\rm Gr}_V^r(\cB_f) \to F_{p+1} {\rm Gr}_V^{r-1}(\cB_f))$ isomorphically onto $F_p \cK_g = \bigcap_{i=1}^r \ker(\partial_{\eta_i} \colon F_p {\rm Gr}_V^r(\cB_g) \to F_{p+1} {\rm Gr}_V^{r-1}(\cB_g))$, because the matrix $A$ is invertible. Thus, we can glue together to give $\cO_Y$-coherent sheaves ${\rm Gr}^F_p \cK_{X}$ supported on $X$.
\end{proof}

\begin{remark} Another way to see that ${\rm Gr}^F_p \cK_f$ is independent of the defining equations is by its identification with ${\rm Gr}^F_{p+r} \mathbf Q_X^H[\dim(X)]$.
\end{remark}

Let ${\rm Gr}^F_p \cQ_X$ be the quotient of ${\rm Gr}^F_p {\rm Gr}^r_V(\cB_{X,Y})$ by ${\rm Gr}^F_p \cK_X$, which is a coherent $\cO_X$-module supported on ${\rm Sing}(X)$. Note that for any $\lambda \notin \bZ_{\geq r}$, the $\cO_X$-module ${\rm Gr}^F_p {\rm Gr}_V^\lambda(\cB_{X,Y})$ is supported on ${\rm Sing}(X)$, too. If $\pi \colon C_XY \to X$ is the projection, we also let $\pi$ denote the vector bundle projection $\pi \colon \pi^{-1}({\rm Sing}(X)) = \Sigma \to \Sigma_X \colonequals {\rm Sing}(X)$.

In the statement of the Corollary, we also study the lowest non-vanishing ${\rm Gr}{\rm DR}$ of the dual $\mathbf D(\varphi_{h,1}(\mathbf Q_{\widetilde{Y}}[\dim(\widetilde{Y})-1])(-\dim Y))$. We state the corollary including the Tate twist by $(-\dim Y)$ because below it is the Tate twisted module that will play a role, and we will see in the proof why the Tate twist naturally arises.

\begin{corollary} \label{cor-lowestHodge} Let $X\subseteq Y$ be a local complete intersection of pure codimension $r$. Let $\widetilde{\alpha}(X) = r+q+\mu$ for some $q\in \bZ_{\geq -1}$ and $\mu \in (0,1]$. Then the lowest non-vanishing index for ${\rm Gr}^F_\bullet {\rm DR}\varphi_{h,\mu}(\mathbf Q_{\widetilde{Y}}^H[\dim(\widetilde{Y})-1])$ occurs at $q+1-\dim Y$, and if $\pi \colon \Sigma \to \Sigma_X$ is the vector bundle projection, then
\[ {\rm Gr}^F_{q+1-\dim Y}{\rm DR} \varphi_{h,\mu}(\mathbf Q_{\widetilde{Y}}^H[\dim(\widetilde{Y})-1]) = \begin{cases} \pi^*({\rm Gr}^F_{q+1-\dim Y} {\rm Gr}_V^{r-1+\mu}(\cB_{X,Y})) & \mu \in (0,1) \\ \pi^*({\rm Gr}^F_{q+1-\dim Y} \cQ_X) & \mu =1\end{cases}.\]

Similarly, if $\mu = 1$, the lowest non-vanishing index for $${\rm Gr}^F_\bullet {\rm DR}(\mathbf D(\varphi_{h,1}(\mathbf Q_{\widetilde{Y}}[\dim(\widetilde{Y})-1])(-\dim Y)))$$ occurs at $q+2-\dim Y$, and we have an isomorphism
\[{\rm Gr}^F_{q+2-\dim Y}{\rm DR}(\mathbf D(\varphi_{h,1}(\mathbf Q_{\widetilde{Y}}[\dim(\widetilde{Y})-1])(-\dim Y))) \cong \pi^*({\rm Gr}_{q+2-\dim Y}^F {\rm Gr}_V^{r-1}(\cB_{X,Y})).\]
\end{corollary}
\begin{proof} The claim that the lowest Hodge piece occurs at $F_{q+1-\dim Y}$ is implied by Theorem \ref{tLCIQ}.

We proceed with a gluing argument like the above. The only difficulty is that the $\cO_{\mathbf C^r}$-action is twisted in the isomorphism constructed above, so we must have care in defining the gluing isomorphisms.

Let $(f) = f_1,\dots, f_r, (g) = g_1,\dots, g_r$ be two regular sequences in $\cO_Y(Y)$ defining the complete intersection $X$. The proof of Theorem \ref{tLCIQ} and equation \eqref{eq-isoLowestHodge} shows that, assuming $\widetilde{\alpha}(X) = r+q+\mu$, the lowest Hodge piece of $\varphi_{h,\mu}(\mathbf Q_{\widetilde{Y}}^H[\dim (\widetilde{Y})-1])$ can be represented using the regular sequence $(f)$ by
\[ {\rm Gr}^F_{q+1-\dim Y} {\rm DR}(\varphi_{h,\mu}(\mathbf Q_{\widetilde{Y}}^H[\dim (\widetilde{Y})-1])) = \begin{cases} {\rm Gr}^F_{q+1-\dim Y} {\rm Gr}_V^{r-1+\mu}(\cB_f)[t_1,\dots, t_r] & \mu \in (0,1) \\ {\rm Gr}^F_{q+1-\dim Y}({\rm Gr}_V^r(\cB_f)/\cK_f)[t_1,\dots, t_r] & \mu =1 \end{cases},\]
and similarly with $\cB_g$ in place of $\cB_f$ and $\eta$ in place of $t$. 

If we choose a matrix $A$ such that $A(f) =(g)$, recall that the fiber coordinates satisfy $\eta_i = \sum_{j=1}^r a_{ij} t_j$. As $A$ is invertible (up to replacing $Y$ by a neighborhood of $X$), we can write $t_i = \sum_{j=1}^r a'_{ij} \eta_j$ for $A^{-1} = A' = (a_{ij}')$ a matrix with entries in $\cO_Y(Y)$.

Then we define an isomorphism for $\mu \in (0,1)$:
\[ \widetilde{\tau}_{f,g} \colon {\rm Gr}^F_{q+1-\dim Y} {\rm Gr}_V^{r-1+\mu}(\cB_f)[t_1,\dots, t_r] \cong {\rm Gr}^F_{q+1-\dim Y} {\rm Gr}_V^{r-1+\mu}(\cB_g)[\eta_1,\dots, \eta_r]\]
and similarly if $\mu = 1$:
\[ \widetilde{\tau}_{f,g} \colon {\rm Gr}^F_{q+1-\dim Y}({\rm Gr}_V^r(\cB_f)/\cK_f)[t_1,\dots, t_r] \cong {\rm Gr}^F_{q+1-\dim Y}({\rm Gr}_V^r(\cB_g)/\cK_g)[\eta_1,\dots, \eta_r],\]
both given by
\[ m t^\alpha \mapsto \tau_{f,g}(m) \prod_{i=1}^r \left(\sum_{j=1}^r a_{ij}' \eta_j \right)^{\alpha_i}.\]

This is clearly $\cO_{Y \times \mathbf C^r}$-linear. As above, it is not hard to check that this is independent of the choice of matrix $A$, and hence see that these isomorphisms satisfy a cocycle condition.

For the last claim, let $s\colon \Sigma_X \to \Sigma$ be the inclusion of the zero section. Then dualizing the vanishing $s^* \varphi_{h,1}(\bQ^H_{\widetilde{Y}}[\dim(\widetilde{Y})-1]) = 0$ (obtained by using \eqref{zero}) and Tate twisting gives us that 
\begin{equation} \label{eq-shriekvanish} s^!(\mathbf D(\varphi_{h,1}(\bQ^H_{\widetilde{Y}}[\dim(\widetilde{Y})-1]))(-\dim Y)) = 0.\end{equation} 

Dualizing the unipotent part of the short exact sequence \eqref{f3} and using that $\bQ^H_Y[\dim Y]$ is pure polarizable of weight $\dim Y$, we get a short exact sequence
\begin{multline*} 0 \to \mathbf D(\varphi_{h,1}(\bQ^H_{\widetilde{Y}}[\dim(\widetilde{Y})-1]))(-\dim Y) \to {\rm Sp}(\bQ_Y^H[\dim(Y)]) \\ \to \mathbf D(\bQ^H_{C_XY}[\dim(X)+r])(-\dim Y) \to 0.\end{multline*}

If we look locally, over an open subset $U \subseteq Y$ such that $U\cap X \subseteq U$ is defined by a regular sequence $f_1,\dots, f_r \in \cO_U(U)$, then we can see that $\mathbf D(\bQ^H_{C_XY}[\dim(X)+r])$ only has non-zero monodromic pieces in degrees $\geq r$ (which is the same property that $\bQ^H_{C_XY}[\dim(X)+r]$ has, essentially because they are pulled back from mixed Hodge modules on $X$). In particular, the $(r-1)$-st monodromic piece of $\mathbf D(\varphi_{h,1}(\bQ^H_{\widetilde{Y}}[\dim(\widetilde{Y})-1]))(-\dim Y)$, as a filtered $\cD$-module, agrees with ${\rm Gr}_V^{r-1}(\cB_f,F)$. The assumption $\widetilde{\alpha}(X) = r+q+1$ (here $\mu =1$) implies $F_{q+1-\dim Y} {\rm Gr}_V^{r-1}(\cB_f) = 0$ but $F_{q+2-\dim Y} {\rm Gr}_V^{r-1}(\cB_f) \neq 0$. Finally, using the vanishing (from \eqref{eq-shriekvanish}) $F_{q+2-\dim Y} \cH^r s^!(\mathbf D(\varphi_{h,1}(\bQ^H_{\widetilde{Y}}[\dim(\widetilde{Y})-1]))(-\dim Y)) = 0$ and the gluing argument above proves the claim.
\end{proof}

%%%%%%%%%%%%%%%%%%%%%%%%%

\subsection{Characteristic classes of lci varieties}\label{charc}
In the notations of Section \ref{defc}, let $$\pi\colon C_XY \to X$$ be the projection, and let $$s=s_0 \colon X \hookrightarrow C_XY$$ be the inclusion of the zero section of $C_XY$. As before, we use $H_*(X)$ to denote either $H_{2*}^{BM}(X;\bQ)$ or $CH_*(X)_\bQ$. The Hirzebruch-Milnor class of the local complete intersection $X \subset Y$ can now be defined by restricting to the zero section the corresponding Hirzebruch-Milnor class of the hypersurface $C_X Y=\{h=0\} \subset \widetilde{Y}$.

\begin{definition}\label{hmci}
The (un-normalized) {\it Hirzebruch-Milnor class} of the local complete intersection $X \subset Y$ is defined by
\begin{equation}\label{unHM}
M_{y*}(X \subset Y):=s^* T_{y*}(\varphi_h \bQ^H_{\widetilde{Y}}) \in H_*(X)[y],
\end{equation}
where $s^*:H_*(C_XY) \to H_{*-r}(X)$ is the Gysin morphism (i.e., the inverse of the isomorphism $\pi^*$, see \cite[Definition 3.3]{Fu}), linearly extended over $\bQ[y]$.
A corresponding (unipotent) class $M^{\{1\}}_{y*}(X \subset Y)$ is defined by using the unipotent vanishing cycles $\varphi_{h,1}$ instead, and similarly we define the class $M^{\{\neq 1\}}_{y*}(X \subset Y)$  by using the non-unipotent vanishing cycles $\varphi_{h,\neq 1}$. 

A normalized version $\widehat{M}_{y*}(X \subset Y)$ of the Hirzebruch-Milnor class is defined by using the Adams-type operations $\Psi_k$, namely, by precomposing $M_{y*}(X \subset Y)$ with the operation $\{\Psi_k\}_{k \geq 0}$, which on a class $\gamma \in H_k(X)$ acts by $\Psi_k(\gamma)=(1+y)^{-k} \cdot \gamma$. Note that since $s^*:H_*(C_XY) \to H_{*-r}(X)$ drops the Chow degree by $r$, an equivalent way to define $\widehat{M}_{y*}(X \subset Y)$ is to replace $T_{y*}$ in \eqref{unHM} by $(1+y)^r \cdot \widehat{T}_{y*}$.
\end{definition}

\begin{remark}\label{rem47}
Using the notations of the previous section, we can define similarly un-normalized and normalized {\it spectral Hirzebruch-Milnor classes}, denoted by $M^{sp}_{t\ast}(X\subset Y)$ and resp. $\widehat{M}^{sp}_{t\ast}(X\subset Y)$, as in Definition \ref{definition: spectral Hirzebruch-Milnor class}. 
\end{remark}

\begin{remark}\label{supp2} 
As above, let
$$\Sigma_X\colonequals\Sing(X)$$ denote the singular locus of $X$. 
Since $\varphi_h \bQ^H_{\widetilde{Y}}$ is supported on $\Sigma:=C_XY|_{\Sigma_X}=\pi^{-1}(\Sigma_X)$, by the functoriality of the transformation $T_{y*}$ for the closed inclusion $\Sigma \subset C_XY$ we can view $T_{y*}(\varphi_h \bQ^H_{\widetilde{Y}})$ as a class localized on $\Sigma$, i.e., in $H_*(\Sigma)[y]$. The functoriality of the Gysin map $s^*$ then implies that we can view the class $M_{y*}(X \subset Y) \in H_*(\Sigma_X)[y]$ as a localized class as well. In particular, 
\begin{center}
$M_{y,i}(X \subset Y)=0 \in H_i(\Sigma_X)[y]$ \ for $i> \dim(\Sing(X))$.
\end{center}
Similar considerations apply to the normalized Hirzebruch-Milnor class $\widehat{M}_{y*}(X \subset Y)$.
\end{remark}

\begin{remark} \label{rmk-compareSpectra} If $X \subseteq Y$ is a complete intersection subvariety of pure codimension $r$ and $x\in X$, then the \textit{reduced spectrum} of \cite{DMS} is given (see \cite[Lemma 4.1]{D23}) by choosing a general point $\xi$ of $\{x\} \times \mathbf C^r \subseteq C_X Y = X\times \bC^r$, with closed embedding $i_\xi \colon \{\xi\} \to C_X Y$, and defining
\[ {\rm Sp}(X,x) = \sum_{\alpha \in \mathbf Q} \overline{m}_{\alpha,x} t^\alpha,\]
where
\[ \overline{m}_{\alpha,x} = \sum_{k\in \mathbf Z} (-1)^k \dim_{\mathbf C} {\rm Gr}^F_{\lceil \alpha \rceil - \dim X -1} \cH^{k-r}i_{\xi}^*(\varphi_{h,\exp(-2\pi i \alpha)}(\mathbf Q_{\widetilde{Y}}^H[\dim(\widetilde{Y}) -1]))\]
\[ = (-1)^{\dim X} \sum_{k\in \mathbf Z} (-1)^k \dim_{\mathbf C} {\rm Gr}^F_{\lceil \alpha \rceil - \dim X - 1} \cH^k i_\xi^*(\varphi_{h,\exp(-2\pi i \alpha)}(\mathbf Q_{\widetilde{Y}}^H)).\]
On the other hand, we have that the coefficient of $t^\alpha$ in ${\rm Sp}'(i_\xi^*(\varphi_h(\bQ^H_{\widetilde{Y}})))$ is
\[ \sum_{k\in \mathbf Z} (-1)^k \dim_{\mathbf C} {\rm Gr}^F_{\lceil -\alpha \rceil} \cH^k i_\xi^*(\varphi_{h,\exp(2\pi i \alpha)}(\bQ^H_{\widetilde{Y}})).\]
We see, then, that if $\iota \colon \mathbf Z[t^{\pm 1/e}]$ is the involution sending $t^\alpha$ to $t^{-\alpha}$, then
\begin{equation} \label{eq-DualSpectrum} {\rm Sp}(X,x) = (-1)^{\dim X} t^{\dim X +1} \iota({\rm Sp}'(i_\xi^*(\varphi_h (\bQ_{\widetilde{Y}}^H)))).\end{equation}
\end{remark}

We also have the following result in the spirit of \cite[Example 3.9]{CMSS}. In the proposition statement, we let ${\rm Sp}'(X,x) = {\rm Sp}'(i_\xi^*\varphi_h(\bQ^H_{\widetilde{Y}}))$. When $X$ has an isolated hypersurface singularity at $x$, for the spectrum as defined in \cite{DMS}, we can take the element $\xi$ to be $(x,1) \in \{x\} \times \mathbf C^r$. 
Moreover, the isomorphism \eqref{eq50monodromy} below gives
\[ i_\xi^* \varphi_{h,\exp(2\pi i \alpha)}(\bQ_{\widetilde{Y}}^H) \cong i_x^* \varphi_{f,\exp(2\pi i \alpha)} (\bQ_Y^H),\]
and so we see that ${\rm Sp'}(X,x)$ as just defined agrees with the definition in Example \ref{ex311}.

\begin{proposition}\label{ptop}
Let $X\subset Y$ be a codimension $r$ local complete intersection in a smooth complex algebraic variety $Y$. Assume for simplicity that $\Sigma_X:=\Sing(X)$ is irreducible.  
Then
\begin{equation}\label{f12}
M_{y*}(X \subset Y)=(1+y)^{\dim(\Sigma_X)+r} \cdot \chi_y(i_\xi^* \varphi_h \bQ^H_{\widetilde{Y}}) \cdot [\Sigma_X] + {\textrm l.o.t.} \in H_{\dim(\Sigma_X)}(\Sigma_X)[y] + \cdots,
\end{equation}
and 
\begin{equation}\label{f12n}
\widehat{M}_{y*}(X \subset Y)=(1+y)^{r} \cdot \chi_y(i_\xi^* \varphi_h \bQ^H_{\widetilde{Y}}) \cdot [\Sigma_X] + {\textrm l.o.t.} \in H_{\dim(\Sigma_X)}(\Sigma_X)[y] + \cdots,
\end{equation}
where $l.o.t$ refers to lower degree homological terms, and $\xi$ is a general point in the fiber of $\pi:C_XY \to X$ over a general (smooth) point  $x\in \Sigma_X$. Similar formulae holds for the spectral classes $M^{sp}_{t\ast}(X\subset Y)$ and $\widehat{M}^{sp}_{t\ast}(X\subset Y)$, upon substituting $t=-y$ and replacing the $\chi_y$-polynomial in \eqref{f12} and resp. \eqref{f12n} with the reduced dual spectrum ${\rm Sp}'(X,x)$. 
\end{proposition}

\begin{proof}
In view of Definition \ref{hmci}, and  since $\varphi_h \bQ^H_{\widetilde{Y}}$ is supported on $\Sigma:=C_XY|_{\Sigma_X}=\pi^{-1}(\Sigma_X)$, with $\Sigma_X:=\Sing(X)$ the singular locus of $X$, 
it suffices to show that
\begin{equation}\label{f13}
\widehat{T}_{y*}(\varphi_h \bQ^H_{\widetilde{Y}})=\chi_y(i_\xi^* \varphi_h \bQ^H_{\widetilde{Y}}) \cdot [\Sigma] + {l.o.t.} \in H_{c+r}(\Sigma)[y] + \cdots,
\end{equation}
where $c=\dim(\Sigma_X)$. 
This follows just like in \cite[Example 3.9]{CMSS}, by noting that the top degree part of $\widehat{T}_{y*}(\varphi_h \bQ^H_{\widetilde{Y}})$ lies in 
\begin{equation}\label{f14}
H_{c+r}(\Sigma)\cong H_{c+r}(\Sigma_{\rm reg}) \cong H^0(\Sigma_{\rm reg}) \cong \bZ,\end{equation}
which is generated by the fundamental class $[\Sigma]$. Here
 $\Sigma_{\rm reg}=\pi^{-1}((\Sigma_X)_{\rm reg})$ is the smooth locus of $\Sigma$, and the second isomorphism in \eqref{f14} is just Poincar\'e duality. 
 (Recall that $H_{c+r}(\Sigma)=H_{top}(\Sigma)$ denotes
as before either the top Borel-Moore homology group or the top Chow group, and note that there is a group isomorphism $CH_{c+r}(\Sigma) \overset{\cong}{\to}  H^{BM}_{2(c+r)}(\Sigma)$.)
 To find the coefficient of $[\Sigma]$ in $\widehat{T}_{y*}(\varphi_h \bQ^H_{\widetilde{Y}})$, we proceed like in \cite[Example 3.9]{CMSS} by taking the stalk of the vanishing cycle complex at a general point $\xi \in \Sigma$. 
  
 Similar considerations apply when using  the spectral class transformation $\widehat{T}_{t*}^{sp}$, in which case the $\chi_y$-polynomial in \eqref{f12} gets replaced with ${\rm Sp'}(i_\xi^* \varphi_h \bQ^H_{\widetilde{Y}}) = {\rm Sp}'(X,x)$.
\end{proof}

\begin{remark}
In general, if $\Sigma_X^i$ is a $c$-dimensional irreducible component of $\Sigma_X=\Sing(X)$, one has canonical arrows
$$H_c(\Sigma_X^i) \to H_c(\Sigma_X) \to H_c((\Sigma_X^i)_{\rm reg})$$
with the first arrow injective. So the argument of the above proof can be applied to each irreducible component of $\Sigma_X$.
\end{remark}

As an immediate consequence of Proposition \ref{ptop}, we have the following.
\begin{corollary}\label{prsp} Assume that the codimension $r$ local complete intersection $X \subset Y$ has only one isolated singularity $\{x\}$. Then 
\begin{equation}\label{icischi}
M_{y*}(X \subset Y)=\widehat{M}_{y*}(X \subset Y)=(1+y)^{r} \cdot \chi_y(i_\xi^* \varphi_h \bQ^H_{\widetilde{Y}}) 
\end{equation}
for $\xi\in\pi^{-1}(x)$ a general point in the fiber of $\pi:C_XY \to X$ over $x$, 
and
\begin{equation}\label{icissp}
M^{sp}_{t\ast}(X\subset Y)= \widehat{M}^{sp}_{t*}(X \subset Y)=(1-t)^r \cdot {\rm Sp}'(X,x).
\end{equation}
\end{corollary}

The first implication in the next proposition is an immediate consequence of Definition \ref{hmci} and Remark \ref{supp2}.

\begin{proposition}
Assume the local complete intersection $X \subset Y$ is smooth. Then $$M_{y*}(X \subset Y)=0 \ \ \  \text{and} \ \ \ \widehat{M}_{y*}(X \subset Y)=0,$$
and similarly, the spectral classes vanish as well. 

Conversely, if the lci $X \subseteq Y$ is singular with singular locus $\Sigma_X$, then the coefficient of the fundamental class $[\Sigma_X]$ in either of the classes $M_{y*}(X \subset Y), \, M^{sp}_{t\ast}(X\subset Y)$ (or their normalized versions) is non-zero. \end{proposition}

\begin{proof}
Since by Remark \ref{supp2} we know that $M_{y*}(X \subset Y) \in H_*(\Sigma_X)[y]$, with $\Sigma_X:=\Sing(X)$, the assertion is immediate. The vanishing of $M^{sp}_{t\ast}(X\subset Y)$ and of the normalized (spectral) classes follows similarly.

We prove that the coefficient of $[\Sigma_X]$ in the class $M^{sp}_{t\ast}(X\subset Y)$ is non-zero, the rest can be shown in a similar fashion or deduced from this claim. By Proposition \ref{ptop}, the coefficient of $[\Sigma_X]$ in $M^{sp}_{t\ast}(X\subset Y)$ is, up to a power of $(1-t)$, given by $ {\rm Sp}'(X,x)$, where $x\in \Sigma_X$ is a general point. Hence the non-vanishing of the coefficient of $[\Sigma_X]$ in $M^{sp}_{t\ast}(X\subset Y)$ is equivalent to ${\rm Sp'}(X,x) \neq 0$, which is equivalent to ${\rm Sp}(X,x) \neq 0$ by Remark \ref{rmk-compareSpectra}.

By \cite[Theorem 2(ii)]{DMS}, this non-vanishing implies the same is true if we take a normal slice to $\Sigma_X$ through $x$. As $x$ is a general point of $\Sigma$, the resulting lci subvariety (in the normal slice) has an isolated singularity at $x$. But the spectrum of an ICIS singularity is always non-zero. Indeed, that spectrum agrees with the Hodge spectrum of the monodromy on the reduced Milnor cohomology, and so the vanishing of the spectrum is equivalent to the vanishing of the reduced Milnor cohomology, i.e., for the Milnor number to be zero. By \cite[Theorem 6.3.10]{Gaf}, the Milnor number is zero if and only if the point is smooth, a contradiction.
\end{proof}

\medskip

Let us next unravel Definition \ref{hmci} in the case when $r=1$, i.e., relate the Hirzebruch-Milnor classes $M_{y*}(X \subset Y)$, $\widehat{M}_{y*}(X \subset Y)$ defined via the deformation to the normal cone  to the Hirzebruch-Milnor classes of the corresponding hypersurface, as introduced in Definition \ref{hms}.
\begin{proposition}\label{pr13}
If $X=\{f=0\}$ with $f\in \cO_Y(Y)$, we have the following identities:
\begin{equation}\label{eq51}
M_{y*}(X \subset Y)=(1+y) \cdot M_{y*}(X)
\end{equation}
and 
\begin{equation}\label{eq51n}
\widehat{M}_{y*}(X \subset Y)=(1+y) \cdot \widehat{M}_{y*}(X).
\end{equation}
Similarly, the spectral classes satisfy
\begin{equation}\label{eq51s}
M^{sp}_{t\ast}(X\subset Y)=(1-t) \cdot M^{sp}_{t\ast}(X).
\end{equation}
and 
\begin{equation}\label{eq51sn}
\widehat{M}^{sp}_{t\ast}(X\subset Y)=(1-t) \cdot \widehat{M}^{sp}_{t\ast}(X).
\end{equation}
\end{proposition}

\begin{proof} 
If $X=\{f=0\}$, then $$\pi \times C(f):C_XY \to X \times \bC$$ is an isomorphism, and let $s_1:X \to C_XY$ be the section corresponding to $id \times \{1\}$ under this isomorphism, i.e., the constant section $z=1$, with $z$ the fiber coordinate. Then, by \cite[(Sp6)]{V}, \begin{equation}s_1^* \circ {\rm Sp}_X=\psi_f,\end{equation} and by applying $s_1^*$ to \eqref{f2} we get that  \begin{equation}\label{eq50} s_1^* (\varphi_h \bQ^H_{\widetilde{Y}})\simeq \varphi_f \bQ^H_Y.\end{equation}
Moreover, using the discussion in Remark \ref{rmkVerdier}, $s_1^*$  identifies the monodromy automorphisms of $\varphi_h$ and $\varphi_f$.
%Recall that there is a discrepancy in the monodromy description we use compared to that in \cite{V}, as they are opposites (see Remark \ref{rmkVerdier} for more information).\\
More precisely, for any $\mu \in (0,1]$, we have the identification (which can also be verified for underlying $\cD$-modules):
 \begin{equation}\label{eq50monodromy} s_1^* (\varphi_{h,\mu} \bQ^H_{\widetilde{Y}})\simeq \varphi_{f,\mu} \bQ^H_Y.\end{equation}

Since $s=s_0$ and $s_1$ are sections of the same (trivial) line bundle $\pi$, they induce the same Gysin morphism, so
\begin{equation}\label{f60u} M_{y*}(X \subset Y):=s^* T_{y*}(\varphi_h \bQ^H_{\widetilde{Y}}) = s_1^* T_{y*}(\varphi_h \bQ^H_{\widetilde{Y}}) .\end{equation}
On the other hand, by \eqref{eq50}, we get
\[ (1+y) \cdot M_{y*}(X)= (1+y) \cdot T_{y*} (s_1^* (\varphi_h \bQ^H_{\widetilde{Y}})) .\]
Hence, in order to prove formula \eqref{eq51}, it suffices to show that for a monodromic mixed Hodge module $M$ on  $C_XY$, one has the following identity
\begin{equation}\label{mon}
s_1^*T_{y*}(M)=(1+y) \cdot T_{y*}(s_1^*M).
\end{equation}

Consider the function $g=z-1:C_XY=X\times \bC \to \bC$, with $z$ the fiber coordinate. Then $g^{-1}(0)$ is identified with $X$ under the inclusion given by the section $s_1$. Sch\"urmann's specialization result for the transformation $\DR_y$ (see \cite{Sp}, or \cite[Section 4]{MS}), coupled with Verdier's specialization result for the Todd class transformation $\td_*$ (cf. \cite{V0,Fu}), yields that
\begin{equation}\label{spH}
s_1^*T_{y*}(M)=(1+y) \cdot T_{y*}(\psi_g(M))
\end{equation}
Thus to prove \eqref{mon} it suffices to show that $s_1^*M \simeq \psi_g(M)$ or, equivalently, that $\varphi_g(M)=0$.

The last claim can be checked on the underlying analytically constructible sheaves, and here we can work on a trivial open disc bundle $X \times D$ around $s_1(X)$ contained in the complement $C_XY\setminus s(X)$ of the zero section. Then a monodromic sheaf $\cF$ on $C_XY$ is (locally) constant on the fibers of the projection $p: X\times D\to X$ so that it is of type $p^*\cG$, from which one gets that $\varphi_{g}\cF=0$.

For the spectral class version, note that Sch\"{u}rmann's result holds more generally for strictly specializable filtered $\cD$-modules. Hence, for any $\mu \in (0,1]$, it applies for the filtered direct summand $\varphi_{h,\mu}(\bQ_{\widetilde{Y}}^H[\dim(\widetilde{Y})-1])$ and we have
\[ s_1^*T_{y*}(\varphi_{h,\mu}(\bQ_{\widetilde{Y}}^H[\dim(\widetilde{Y})-1]))=(1+y) \cdot T_{y*}(\psi_g(\varphi_{h,\mu}(\bQ_{\widetilde{Y}}^H[\dim(\widetilde{Y})-1]))),\]
but then using that $\varphi_{h,\mu}(\bQ_{\widetilde{Y}}^H[\dim(\widetilde{Y})-1])$ is (locally) monodromic, we get that the right hand side is equal to $(1+y) T_{y*}(s_1^*(\varphi_{h,\mu}(\bQ_{\widetilde{Y}}^H[\dim(\widetilde{Y})-1])))$. Evaluating at $y = -t$ and using \eqref{eq50monodromy}, we have the equality
\[s_1^*T_{(-t)*}(\varphi_{h,\mu}(\bQ_{\widetilde{Y}}^H[\dim(\widetilde{Y})-1])) = (1-t) T_{(-t)*}(s_1^*(\varphi_{h,\mu}(\bQ_{\widetilde{Y}}^H[\dim(\widetilde{Y})-1])))\]
\[ = (1-t) T_{(-t)*}(\varphi_{f,\mu}(\bQ_Y^H[\dim Y])).\]
Multiplying by $t^{1-\mu}$ for $\mu \in (0,1)$ (because $\varphi_{*,\mu}$ corresponds to eigenvalue $\exp(-2\pi i \mu)$ for $* = h, f$) or by $t^0$ if $\mu =1$, and summing over all $\mu \in (0,1]$, we get the desired equality \eqref{eq51s}.

To prove formula \eqref{eq51n}, we proceed similarly, the only changes appearing in formula \eqref{f60u}, which gets replaced by
\begin{equation}\label{f60n} \widehat{M}_{y*}(X \subset Y):=(1+y)\cdot  s_1^* \widehat{T}_{y*}(\varphi_h \bQ^H_{\widetilde{Y}}) ,\end{equation}
and also formula \eqref{spH} gets replaced by (see \cite{Sp}, or \cite[Section 4]{MS})
\begin{equation}\label{spHn}
s_1^*\widehat{T}_{y*}(M)=\widehat{T}_{y*}(\psi_g(M)).
\end{equation}
Formula \eqref{eq51sn} follows similarly.
\end{proof}

\begin{remark}
In view of Corollary \ref{prsp} and Example \ref{ex311}, Proposition \ref{pr13} provides a characteristic class (and higher homological) generalization of the fact that, if the hypersurface $X=\{f=0\} \subset Y$ has only an isolated singularity, then the two notions of Hodge spectrum ${\rm Sp'}(X,x)$ (one defined as in Example \ref{ex311} via the vanishing cycles of $f$, and the other defined as in \cite{DMS} via the deformation to the normal cone) agree (cf. also Remark \ref{rmk-compareSpectra}). 
\end{remark}

%%%%%%%%%%%%%%%%%%%%%%

\subsection{Spectral classes, minimal exponent, and Hodge spectrum}\label{sec: spectralclasses} 
In this section we relate the spectral classes of a local complete intersection $X$ to the minimal exponent  $\widetilde{\alpha}(X)$. As a consequence, we discuss results about the Hodge spectrum. For simplicity, we formulate our results in terms of the un-normalized classes.

Recall that for $\alpha \in \bQ$, we denote the coefficient of $t^\alpha$ in the spectral Hirzebruch-Milnor class $M^{sp}_{t\ast}(X\subset Y)$ by 
 $M^{sp}_{t\ast}(X\subset Y)|_{t^{\alpha}}\in H_{\ast}(\Sing(X))$. 
 
 The main result of this section is the following.

\begin{thm} \label{thm-mainHighCoefficients}
Let $X\subset Y$ be a codimension $r$ local complete intersection in a smooth complex algebraic variety $Y$.

If $\widetilde{\alpha}(X) \geq r+q+\mu$ for some $q \in \bZ_{\geq -1}$ and $\mu \in (0,1]$, then
\begin{equation} M^{sp}_{t\ast}(X\subset Y)\vert_{t^{\alpha}}  = 0 \text{ for all } \alpha > \dim Y - q - \mu.\end{equation}

If ${\rm Sing}(X)$ is projective and ${\rm lct}(X) > r-1$, then the converse holds.
\end{thm}
\begin{proof} By Theorem \ref{tLCIHighSings} and Theorem \ref{tLCIQ}, we know that $\widetilde{\alpha}(X) \geq r+q+\mu$ if and only if \begin{equation}\label{f30}
\begin{cases} 
{\gr}^F_{p - \dim Y} {\rm DR}(\varphi_{h,\lambda}(\mathbf Q^H_{\widetilde{Y}}[\dim(\widetilde{Y})-1])) = 0 \text{ for all } \lambda \in (0,\mu), \quad p \leq q +1, \\
 {\gr}^F_{p - \dim Y} {\rm DR}(\varphi_{h,\lambda}(\mathbf Q^H_{\widetilde{Y}}[\dim(\widetilde{Y})-1])) = 0 \text{ for all } \lambda \in [\mu,1], \quad p \leq q
\end{cases}
\end{equation}
Recall now that for an integer $j \in \bZ$ and $\beta \in [0,1)$, we have the equality
\[ M^{sp}_{t\ast}(X\subset Y)\vert_{t^{j+\beta}} = s^* \td_*({\rm Gr}^F_{-j}{\rm DR}(\varphi_{h,\exp(2\pi i \beta)}(\mathbf Q^H_{\widetilde{Y}}))) . \]
Thus, we have the equality (using that $\varphi_{h,\lambda}$ corresponds to eigenvalue $\exp(-2\pi i \lambda)$):
\[ s^* \td_*({\rm Gr}^F_{p-\dim Y}{\rm DR}(\varphi_{h,\lambda}(\mathbf Q^H_{\widetilde{Y}}))) = \begin{cases} M^{sp}_{t\ast}(X\subset Y)\vert_{t^{\dim Y - p + (1-\lambda)}} & \lambda \in (0,1) \\ M^{sp}_{t\ast}(X\subset Y)\vert_{t^{\dim Y - p}} & \lambda = 1  \end{cases}.\]
Finally, this gives the vanishings
\[\begin{cases} M^{sp}_{t\ast}(X\subset Y)\vert_{t^{\dim Y +1 - p - \lambda}} = 0 & \lambda \in (0,\mu), p \leq q+1 \\ M^{sp}_{t\ast}(X\subset Y)\vert_{t^{\dim Y +1 - p - \lambda}} = 0 & \lambda \in [\mu,1), \quad p \leq q\\ M^{sp}_{t\ast}(X\subset Y)\vert_{t^{\dim Y - p}} = 0 & \lambda = 1 , \quad p \leq q \end{cases},\]
which proves the first claim.

Suppose now that ${\rm Sing}(X)$ is projective and that $M^{sp}_{t\ast}(X\subset Y)\vert_{t^{\alpha}} =0$ for all $\alpha > \dim Y - q -\mu$. Assume to the contrary that $\widetilde{\alpha}(X) < r+q+\mu$. Write $\widetilde{\alpha}(X) = r+q'+\mu'$, with $q' \in \bZ_{\geq -1}$ and $\mu' \in (0,1]$, so that $q'+\mu' < q+\mu$. There are two cases: either $q' = q$ and $\mu' < \mu$ or $q' \leq q-1$.

By Theorem \ref{tLCIQ}, we have the equality 
\[\widetilde{\alpha}(X) = r -1 + \min\{p+\lambda \mid \lambda \in (0,1], {\rm Gr}^F_{p-\dim Y}{\rm DR}\varphi_{h,\lambda}(\bQ^H_{\widetilde{Y}}[\dim(\widetilde{Y})-1]) \neq 0\},\] and so
\[ \min\{p+\lambda \mid  \lambda \in (0,1], {\rm Gr}^F_{p-\dim Y}{\rm DR}\varphi_{h,\lambda}(\bQ^H_{\widetilde{Y}}[\dim(\widetilde{Y})-1]) \neq 0\} = q'+\mu'+1,\]
which gives that the first non-zero ${\rm Gr}^F_\bullet {\rm DR}$ occurs at index $q'+1-\dim Y$:
\[ {\rm Gr}^F_{q'+1-\dim Y}{\rm DR} \varphi_{h,\mu'}(\mathbf Q_{\widetilde{Y}}^H[\dim(\widetilde{Y})-1]) \neq 0.\]
By the isomorphism \eqref{eq-isoLowestHodge}, this is a coherent sheaf of the form $\pi^*(\cF)$ for $\cF$ some coherent sheaf on $\Sigma_X={\rm Sing}(X)$. Since $\pi:\Sigma \to \Sigma_X$ is a vector bundle, the exact pullback functor $\pi^*$ on coherent sheaves induces an isomorphism on Grothendieck groups 
\[ \pi^*:K_0(\Sigma_X) \overset{\sim}{\longrightarrow} K_0(\Sigma) . \]
Since $\cF$ is a non-zero coherent sheaf on the projective variety $\Sigma_X$, its Grothendieck class $[\cF] \in K_0(\Sigma_X)_\bQ=K_0(\Sigma_X)\otimes_\bZ \bQ$ does not vanish. This follows, e.g., from \cite[Proposition 1]{MSY23}, since $\td_*([\cF])\neq 0 \in H_*(\Sigma_X)$ and 
\begin{equation}\label{tdi} \td_*:K_0(\Sigma_X)_\bQ \overset{\sim}{\longrightarrow} \bigoplus_i CH_i(\Sigma_X)_\bQ\end{equation}
is an isomorphism (cf. \cite[Corollary 18.3.2]{Fu}).
So $[\pi^*\cF]=\pi^*[\cF]\neq 0 \in K_0(\Sigma)_\bQ$. Applying the Todd transformation and using the corresponding isomorphism \eqref{tdi} for $\Sigma$, we get that $\td_*(\pi^*\cF)\neq 0$. Finally, since $s^*$ is the inverse of the isomorphism $\pi^*$ on Chow (homology) groups, we conclude that
\[ s^* {\rm td}_*({\rm Gr}^F_{q'+1-\dim Y} {\rm DR}(\varphi_{h,\mu'}(\mathbf Q_{\widetilde{Y}}^H[\dim(\widetilde{Y})-1])) = s^* {\rm td}_*(\pi^*(\cF)) \neq 0.\]

Note that the left hand side is the coefficient of $\begin{cases} t^{\dim Y - q'-1 +(1-\mu')} & \mu' \in (0,1) \\ t^{\dim Y - q' -1} & \mu' = 1\end{cases}$. 

In the second case, we must have $q' \leq q-1$, and so we get $\dim Y - q' -1 \geq \dim Y - (q-1) -1 > \dim Y - q - \mu$, using that $\mu \in (0,1]$, which contradicts our assumption.

In the first case, either $q' = q$ and $\mu' < \mu$, so that $\dim Y - q' -\mu' > \dim Y - q' - \mu =\dim Y - q - \mu$, contradicting our vanishing assumption. Finally, if $q' \leq q-1$, we get $\dim Y - q' - \mu' \geq \dim Y - q +1 - \mu' > \dim Y - q - \mu$, as both $\mu,\mu' \in (0,1]$.
\end{proof}

\begin{remark} Let $X \subseteq Y$ be a local complete intersection of pure codimension $r$ with isolated singularity at $x\in X$ and assume ${\rm lct}(X) > r-1$. Let $${\rm Sp}(X,x)=\sum_{\alpha \in \bQ} \overline{m}_{\alpha,x} t^\alpha$$ be the spectrum of $X$ at $x$, as defined in Remark \ref{rmk-compareSpectra} above.
Then \cite{D23} gives equality
\[ \min\{\alpha \mid \overline{m}_{\alpha,x} \neq 0\} = \widetilde{\alpha}(X) - r +1.\]
Theorem \ref{thm-mainHighCoefficients} gives another proof of this equality (compare with Remark \ref{rem314}).
\end{remark}

\begin{corollary}\label{cl} 
Assume that the local complete intersection $X \subset Y$ has only one isolated singularity $\{x\}$ and satisfies ${\rm lct}(X) > r-1$. Let $\widetilde{\alpha}(X) = r+q+\mu$ with $q \in \bZ_{\geq -1}$ and $\mu \in (0,1]$. Then 
\[ \min\{\alpha \mid \overline{m}_{\alpha,x} \neq 0\} = \widetilde{\alpha}(X) - r +1.\]

In particular, $X$ has $k$-Du Bois singularities if and only if
\begin{equation}\label{f40sp}
\overline{m}_{\alpha,x}=0 \ \ \text{for  all} \ \alpha <k+1.
\end{equation}
Finally, $X$ has $k$-rational singularities if and only if
\begin{equation} \label{f41sp}
\overline{m}_{\alpha,x}\ \ \text{for  all} \ \alpha \leq k+1.
\end{equation}
\end{corollary}
\begin{proof} As mentioned above, we have the relation
\[ {\rm Sp}(X,x) = (-1)^{\dim X} t^{\dim X +1} \iota({\rm Sp}'(X,x)).\]
where recall that $\iota \colon \bZ[t^{1/e},t^{1/e}] \to \bZ[t^{1/e},t^{-1/e}]$ is the involution sending $t^{1/e}$ to $t^{-1/e}$. We can ignore the sign $(-1)^{\dim X}$ as we are only interested in (non)-vanishing.

If we write ${\rm Sp}'(X,x) = \sum_{\alpha \in \bQ} \overline{m}_{\alpha,x}' t^\alpha$, then we have $\overline{m}_{\alpha,x} = \overline{m}_{\dim X +1-\alpha,x}'$. Thus, the desired implication is
\[ \widetilde{\alpha}(X) = r+q+\mu \implies \overline{m}_{\alpha,x}' = 0 \text{ for all } \alpha > \dim(X) - q - \mu.\]

Let $\chi = \max\{\alpha \mid \overline{m}_{\alpha,x}'\neq 0\}$. By expanding $(1-t)^r$ in Proposition \ref{prsp}, we see then that the maximal non-zero degree of $t$ in $M^{sp}_{t\ast}(X\subset Y)$ is $\chi +r$. Theorem \ref{thm-mainHighCoefficients} then gives the inequality $\chi +r \leq \dim(Y) - q - \mu$, so we get $\chi \leq \dim(X) -q - \mu$, as desired.

The last two claims follow by taking $q = k-1, \mu = 1$ and $q = k, 0 < \mu \ll 1$, respectively.
\end{proof}

%%%%%%%%%%%%%%%%%%%%%%%%%%

\subsection{Spectral classes via Hodge-theoretic duality}\label{dua}
We begin with the main result of this subsection, which relates the lower spectral classes to the minimal exponent.

\begin{thm}\label{tlowCoefficients}
Let $X\subset Y$ be a codimension $r$ local complete intersection in a smooth complex algebraic variety $Y$. Assume $\widetilde{\alpha}(X) \geq r+q+\mu$ with $q\in\bZ_{\geq -1}$ and $\mu \in (0,1]$. Then
\[M^{sp}_{t\ast}(X\subset Y)\vert_{t^{\alpha}}=0 \ \text{ for  all } \alpha < q+1+\mu,\]
and the converse holds if ${\rm lct}(X) > r-1$ and ${\rm Sing}(X)$ is projective.
\end{thm}
\begin{proof} Recall that, for $\lambda \in [0,1)$, the coefficient $M^{sp}_{t\ast}(X\subset Y)\vert_{t^{p+\lambda}}$ is 
\[s^* {\rm td}_*([{\rm Gr}^F_{-p} {\rm DR}(\varphi_{h,1-\lambda}(\bQ^H_{\widetilde{Y}}[\dim(\widetilde{Y})-1]))]).\]

 Assume $\widetilde{\alpha}(X) \geq r+q+\mu$, then we will show 
 \[ {\rm Gr}^F_{-p}{\rm DR}(\varphi_{h,1-\lambda}(\bQ_{\widetilde{Y}}^H[\dim(\widetilde{Y})-1])) = 0 \text{ for } \begin{cases} p \leq q & \lambda \in [\mu,1) \\ p \leq q+1 & \lambda \in [0,\mu)\end{cases}.\] By duality, this vanishing is equivalent to the following: for $\lambda \in (0,1)$, we want
 \[ {\rm Gr}^F_{p-\dim Y}{\rm DR}(\mathbf D(\varphi_{h,\lambda}(\bQ_{\widetilde{Y}}^H[\dim(\widetilde{Y})-1]))(-\dim Y)) = 0 \text{ for } \begin{cases} p \leq q & \lambda \in [\mu,1) \\ p \leq q+1 & \lambda \in (0,\mu) \end{cases}\]
 and for $\lambda = 0$, we want
 \[ {\rm Gr}^F_{p-\dim Y}{\rm DR}(\mathbf D(\varphi_{h,1}(\bQ_{\widetilde{Y}}^H[\dim(\widetilde{Y})-1]))(-\dim Y)) = 0 \text{ for } p \leq q+1.\]

Under the assumption of the minimal exponent, by the vanishing in \eqref{f30}, we get
\begin{equation}
\begin{cases} 
 {\gr}^F_{p - \dim Y} {\rm DR}(\varphi_{h,\lambda}(\mathbf Q^H_{\widetilde{Y}}[\dim(\widetilde{Y})-1])) = 0 \text{ for all } \lambda \in [\mu,1), \quad p \leq q \\ 
{\gr}^F_{p - \dim Y} {\rm DR}(\varphi_{h,\lambda}(\mathbf Q^H_{\widetilde{Y}}[\dim(\widetilde{Y})-1])) = 0 \text{ for all } \lambda \in (0,\mu), \quad p \leq q +1,
\end{cases}
\end{equation}

This immediately gives us the vanishings we want for $\lambda \in (0,1)$. Indeed, using that ${\rm Sp}_X(-)$ commutes with duality $\bD$ and the isomorphism
\[ {\rm Sp}_{X,\lambda}(\mathbf Q_Y^H[\dim(Y)]) \cong \varphi_{h,\lambda}(\mathbf Q_{\widetilde{Y}}^H[\dim(\widetilde{Y})-1]) \text{ for } \lambda \in (0,1)\]
we conclude that there is an isomorphism
\begin{equation} \label{eq-nonunipDual} \mathbf D(\varphi_{h,\lambda}(\mathbf Q_{\widetilde{Y}}^H[\dim(\widetilde{Y})-1]))(-\dim Y)\cong \varphi_{h,1-\lambda}(\mathbf Q_{\widetilde{Y}}^H[\dim(\widetilde{Y})-1]),\end{equation}
where the Tate twist on the left hand side uses the fact that $\bQ_Y^H[\dim Y]$ is a pure Hodge module of weight $\dim Y$.

For the unipotent part, towards contradiction let $p\leq q+1$ be minimal such that 
\[{\rm Gr}^F_{p-\dim Y} {\rm DR}(\mathbf D(\varphi_{h,1}(\bQ^H_{\widetilde{Y}}[\dim(\widetilde{Y})-1]))(-\dim Y)) \neq 0.\]
Then, by the end of the proof of Corollary \ref{cor-lowestHodge}, there exists some $U\subseteq Y$ with $f_1,\dots, f_r$ a regular sequence defining $X\cap U \subseteq U$ such that
\[ {\rm Gr}^F_{p-\dim Y} {\rm Gr}_V^{r-1}(\cB_f) \neq 0,\]
but the assumption $\widetilde{\alpha}(X) \geq r+q + \mu$ implies by Theorem \ref{tLCI} that $F_{q+1-\dim Y} {\rm Gr}_V^{r-1}(\cB_f) = 0$, and so the inequality $p\leq q+1$ gives a contradiction.

For the converse, assume ${\rm Sing}(X)$ is projective and $M^{sp}_{t\ast}(X\subset Y)\vert_{t^{\alpha}} = 0$ for all $\alpha < q+1+\mu$, but assume toward contradiction that $r-1 < \widetilde{\alpha}(X) = r+p+\lambda < r+q+\mu$. Here we take $p \in \Z_{\geq -1}$ and $\lambda \in (0,1]$.

First, if $\lambda \in (0,1)$, then by Corollary \ref{cor-lowestHodge} we have an isomorphism
\[ {\rm Gr}^F_{p+1-\dim Y}{\rm DR}(\varphi_{h,\lambda}(\bQ^H_{\widetilde{Y}}[\dim(\widetilde{Y})-1]))) = \pi^*(\cF) \neq 0,\]
where $\cF = {\rm Gr}^F_{p+1-\dim Y} {\rm Gr}_V^{r-1+\lambda}(\cB_{X,Y})$, and because $p+\lambda < q+\mu$ we have either that $\lambda < \mu$ and $p\leq q$ or $\lambda \in [\mu,1)$ and $p \leq q-1$, and in any case, $p+1+\lambda < q+1+\mu$. Applying Grothendieck duality $\mathbb D_{\rm coh} = R\cH om_{\cO}(-,\omega_\Sigma^\bullet)$, which is known to satisfy the following isomorphism on mixed Hodge modules (\cite[Sect. 2.4]{Saito88}):
\[ \mathbb D_{\rm coh} {\rm Gr}^F_{\bullet}{\rm DR}(-) \cong {\rm Gr}^F_{-\bullet}{\rm DR}(\mathbf D(-)),\]
and using the duality relation \eqref{eq-nonunipDual}, we get
\[ {\rm Gr}^F_{-(p+1)} {\rm DR}(\varphi_{h,1-\lambda}(\bQ^H_{\widetilde{Y}}[\dim(\widetilde{Y})-1])) = \mathbb D_{\rm coh}(\pi^*(\cF)).\]

By the argument of Theorem \ref{thm-mainHighCoefficients}, we know that ${\rm td}_*([\pi^*(\cF)]) \neq 0$. Hence, by \cite[Example 18.3.19]{Fu}, we conclude that ${\rm td}_*([\mathbb D_{\rm coh}(\pi^*(\cF))]) \neq 0$. Finally, as $s^*$ is injective, we get the coefficient $M^{sp}_{t\ast}(X\subset Y)\vert_{t^{p+1+\lambda}} \neq 0$, which using that $p+1+\lambda < q+1+\mu$ gives a contradiction.

If $\lambda = 1$, then by the last statement of Corollary \ref{cor-lowestHodge}, we have an isomorphism
\[ {\rm Gr}^F_{p+2-\dim Y}{\rm DR}(\mathbf D(\varphi_{h,1}(\bQ^H_{\widetilde{Y}}[\dim(\widetilde{Y})-1]))(-\dim Y)) = \pi^*(\cF) \neq 0,\]
where $\cF = {\rm Gr}^F_{p+2-\dim Y} {\rm Gr}_V^{r-1}(\cB_{X,Y})$. As $p+\lambda < q+\mu$ and $\lambda = 1$, we must have $\mu =1$ and so $p+2 \leq q+1$. If we apply Grothendieck duality to both sides, we get
\[ {\rm Gr}^F_{-(p+2)} {\rm DR}(\varphi_{h,1}(\bQ^H_{\widetilde{Y}}[\dim(\widetilde{Y})-1])) \cong \mathbb D_{\rm coh}(\pi^*(\cF)).\]
By assumption, if we apply $s^* {\rm td}_*$ to the left hand side of the above isomorphism, we get the coefficient $M^{sp}_{t\ast}(X\subset Y)\vert_{t^{p+2}} \neq 0$, but by the inequality $p+2\leq q+1$, we get a contradiction.
\end{proof}

\begin{corollary} Assume that the local complete intersection $X \subset Y$ has only one isolated singularity $\{x\}$ and satisfies ${\rm lct}(X) > r-1$. Let $\widetilde{\alpha}(X) = r+q+\mu$ with $q \in \bZ_{\geq -1}$ and $\mu \in (0,1]$. Then 
\[ \max\{\alpha \in \bQ \mid \overline{m}_{\alpha,x} \neq 0\} = \dim(X) - q - \mu .\]
\end{corollary}
\begin{proof} By Proposition \ref{ptop}, we have
\[ M_{t*}^{\rm sp}(X\subseteq Y) = {\rm Sp}'(X,x) = \sum_{\alpha \in \bQ} \overline{m}'_{\alpha,x} t^\alpha\]
in this setting. Thus, Theorem \ref{tlowCoefficients} gives equality (using that the converse holds in the isolated singularities setting)
\[ \min \{ \alpha \mid \overline{m}_{\alpha,x}' \neq 0\} = q+1+\mu.\]

Using the identification \eqref{eq-DualSpectrum}, this says that
\[ \max\{\alpha \mid \overline{m}_{\alpha,x} \neq 0\} = (\dim X + 1) - (q+1+\mu) = \dim(X) - q - \mu,\]
as desired.
\end{proof}

We conclude this section by relating the unipotent spectral classes to a weakening of the $\bQ$-homology manifold condition, recently introduced in \cite{PP,DOR}.

On any reduced complex algebraic variety $Z$, there is a natural morphism (defined in \cite[(4.5.12)]{Saito90}) in $D^b({\rm MHM}(Z))$
\[ \psi_Z \colon \bQ_Z^H[\dim Z] \to \mathbf D_Z(\bQ_Z^H[\dim Z])(-\dim Z),\]
which is a quasi-isomorphism if $Z$ is smooth. More generally, the morphism $\psi_Z$ is a quasi-isomorphism if and only if $Z$ is a $\bQ$-homology manifold, the underlying morphism of $\bQ$-perverse sheaves being a sheaf-theoretic incarnation of the Poincar\'{e} duality morphism.

The papers \cite{PP,DOR} introduce a natural Hodge-theoretic weakening of the $\bQ$-homology manifold condition, which is defined by applying ${\rm Gr}^F_{-p}{\rm DR}$ to the morphism $\psi_Z$ and asking whether the morphism ${\rm Gr}^F_{-p}{\rm DR}(\psi_Z)$ is a quasi-isomorphism in $D^b_{\rm coh}(\cO_Z)$. We define
\[ {\rm HRH}(Z) = {\rm sup}\left\{k \mid {\rm Gr}^F_{-p}{\rm DR}(\psi_Z) \text{ is a quasi-isomorphism for all } p \leq k\right\},\]
with the convention that ${\rm HRH}(Z) = -1$ if ${\rm Gr}^F_0 {\rm DR}(\psi_Z)$ is not a quasi-isomorphism.

It is not hard to see that if $\pi \colon Z' \to Z$ is a smooth morphism, then ${\rm HRH}(Z) = {\rm HRH}(Z')$, see \cite[Lemma 6.10]{DOR} for a proof.

Essentially by definition, if $X$ is a variety with lci singularities, then for any $k\geq 0$, we have
\begin{equation} \label{eq-differenceDBRatl} X \text{ is } k\text{-rational} \iff X \text{ is }k\text{-Du Bois and } {\rm HRH}(X) \geq k.\end{equation}

If $M \in {\rm MHM}(Z)$ is a mixed Hodge module on a (possibly singular) complex variety $Z$, we define $$p(M) = \min\{q \mid {\rm Gr}^F_q {\rm DR}(M) \text{ is not acyclic}\}.$$ 

In the embedded lci setting $X\subseteq Y$, we have the following. 

\begin{thm} \label{tHRHInequality} $(a)$ (\cite[Theorem I]{DOR}) Let $X\subseteq Y$ be an lci subvariety. Then
\[ {\rm HRH}(X) \geq p(\varphi_{h,1}(\bQ_{\widetilde{Y}}^H[\dim(\widetilde{Y})-1])) + \dim Y - 1.\]

$(b)$ (\cite[Theorem H]{DOR}) Let $X\subseteq Y$ be a hypersurface. Then
\[ {\rm HRH}(X) = p(\varphi_{h,1}(\bQ_{\widetilde{Y}}^H[\dim(\widetilde{Y})-1])) + \dim Y - 1,\]
and if $X =\{f=0\} \subseteq Y$, then
\[ {\rm HRH}(X) = p(\varphi_{f,1}(\bQ^H_{Y}[\dim(Y)])) + \dim Y -2.\]
\end{thm}

A relation between the invariant ${\rm HRH}(X)$ defined above and the unipotent Hirzebruch-Milnor classes is given by the next result, however, we can only handle the hypersurface case. The difficulty with the general lci case is that, in Corollary \ref{cor-lowestHodge}, one only knows that the first non-zero ${\rm Gr}{\rm DR}$ is a pull-back along $\pi \colon \Sigma \to \Sigma_X$ in the case where that lowest Hodge piece actually corresponds to the minimal exponent. In particular, if the minimal exponent is not an integer, then we do not have any precise understanding of ${\rm Gr}^F_\bullet {\rm DR}(\varphi_{h,1}(\bQ_{\widetilde{Y}}^H[\dim(\widetilde{Y})-1]))$. So even when $\Sigma$ is projective, we cannot conclude vanishing, except in the hypersurface case where we can work on $\Sigma_X$ itself.

\begin{thm} \label{thm-HRH} Let $X\subseteq Y$ be a hypersurface in a smooth complex algebraic variety $Y$. Then 
\[ \min \{ p \in \bZ \mid [M_{y*}^{\{1\}}(X\subseteq Y)]_p \neq 0\} \geq  {\rm HRH}(X) +1,\]
and if ${\rm Sing}(X)$ is projective, then equality holds.
\end{thm}
\begin{proof} We have equality by Proposition \ref{pr13}:
\[ M_{y*}^{\{1\}}(X\subseteq Y) = (1-y) M_{y*}^{\{1\}}(X),\]
and so it suffices to prove the claim with $[M_{y*}^{\{1\}}(X)]_p$.

For this, we use \cite[Theorem 4]{MSY23}, which says that 
\[ [M_{y*}^{\{1\}}(X)]_p = 0 \text{ for all } p < \widetilde{\alpha}^{\{1\}}(X),\]
with equality if ${\rm Sing}(X)$ is projective. Here $\widetilde{\alpha}^{\{1\}}(X) = \min p({\rm Gr}_V^0(\cB_{f,U})) + \dim Y$, where the minimum is over all open subsets $U\subseteq Y$ with $f\in \cO_U(U)$ such that $X\cap U = \{f=0\} \subseteq U$. 

Putting this together with the last claim of Theorem \ref{tHRHInequality} above (applied locally), we get
\[ \min\{p \in \bZ \mid [M_{y*}^{\{1\}}(X)]_p \neq 0\} \geq {\rm HRH}(X) +1,\]
with equality if ${\rm Sing}(X)$ is projective, as desired.
\end{proof}

%%%%%%%%%%%%%%%%%%%%%%%%%%

\subsection{Higher singularities via Hirzebruch-Milnor classes}\label{hshmc}
In this section, we use the fact that a corresponding formula \eqref{MHsp} holds for the (un-normalized) Hirzebruch-Milnor classes in order to relate higher (Du Bois and rational) singularities to vanishing properties of these Hirzebruch-Milnor classes.

\begin{notation}
For an integer $p \geq 0$ we denote as before by  
$[M_{y\ast }(X\subset Y)]_{p}$
the coefficient of $y^p$ in $M_{y\ast}(X\subset Y)$, and similarly for  $[M^{\{1\}}_{y*}(X \subset Y)]_p$ and  $[M^{\{\neq 1\}}_{y*}(X \subset Y)]_p$.
\end{notation}
We then have the following results.
\begin{thm} \label{cor-higherSings}
Let $X\subset Y$ be a codimension $r$ local complete intersection in a smooth complex algebraic variety $Y$. If $X$ has $k$-Du Bois singularities, then
\begin{equation}\label{f29}
\begin{cases}
[M^{\{\neq 1\}}_{y*}(X \subset Y)]_p=0  & \text{for  all} \ p\geq \dim(Y)-k,  \\
[M^{\{1\}}_{y*}(X \subset Y)]_p=0  &  \text{for  all} \  p\geq \dim(Y)+1-k,
\end{cases}
\end{equation}
and the converse holds if ${\rm Sing}(X)$ is projective and ${\rm lct}(X) > r-1$.

If $X$ has $k$-rational singularities, then
\begin{equation} \label{f28}
[M_{y\ast}(X\subset Y)]_{p}=0 \ \ \text{for  all}  \ p \geq \dim(Y)- k,
\end{equation}
and the converse holds if ${\rm Sing}(X)$ is projective and ${\rm lct}(X) > r-1$.
\end{thm}

\begin{proof} The starting point is the equality
\[ [M_{(-y)\ast}(X)]_{p}=\bigoplus_{\alpha \in \bQ\cap[0,1)} M^{sp}_{t\ast}(X)|_{t^{p+\alpha}}.\]

We have that $X$ has $k$-Du Bois singularities if and only if $\widetilde{\alpha}(X) \geq r+k$. By taking $q=k-1$ and $\mu = 1$ in Theorem \ref{thm-mainHighCoefficients}, this implies the desired vanishings.

Similarly, $X$ has $k$-rational singularities if and only if $\widetilde{\alpha}(X) > r+k$, and so the last claim follows in the same way by taking $q =k$ and $0 < \mu \ll 1$.
\end{proof}

\begin{thm} \label{thm-mainLowCoefficients}
Let $X\subset Y$ be a codimension $r$ local complete intersection in a smooth complex algebraic variety $Y$. If $X$ has $k$-Du Bois singularities, then
\begin{equation}\label{f40}
[M_{y*}(X \subset Y)]_p=0 \  \text{for  all} \ p\leq k.
\end{equation}
If $X$ has $k$-rational singularities, then \eqref{f40} holds and, moreover,
\begin{equation} \label{f41}
[M_{y\ast}^{\{1\}}(X\subset Y)]_{k+1}=0.
\end{equation}

If ${\rm lct}(X) > r-1$ and ${\rm Sing}(X)$ is projective, then the converse implications hold.
\end{thm}

\begin{proof}
    The assertion follows by using Theorem \ref{tlowCoefficients}, or more specifically,  by taking $q=k-1, \mu =1$ in the $k$-Du Bois case and $q = k,0 < \mu \ll 1$ in the $k$-rational case.
\end{proof}

\begin{remark}
Let us assume that $X=\{f=0\}$ is a globally defined hypersurface in the smooth variety $Y$. Proposition \ref{pr13} yields in this case that
\[[M_{y\ast}(X\subset Y)]_{p}=[M_{y\ast}(X)]_{p} + [M_{y\ast}(X)]_{p-1},\]
and so we can see the vanishings in Theorem \ref{thm-mainLowCoefficients} from Corollary \ref{thm: main 1}, if $X$ has $k$-Du Bois singularities and Corollary \ref{thm: main 2}, when $X$ has $k$-rational singularities.
\end{remark}

%%%%%%%%%%%%%%%%%%%%%%%%%%%%%%%%%%%%%%%%%%%%%%

\end{document}